\documentclass[11pt,a4paper]{amsart}
\usepackage{amsmath, amsthm}              
\usepackage{pgfplots,pdflscape,comment}
\usepackage{amssymb,bm}
\usepackage{amsthm, wasysym,url}

\newtheorem{thm}{Theorem}[section]

\newtheorem{lemma}[thm]{Lemma}

\newtheorem{example}[thm]{Example}
\newtheorem{definition}[thm]{Definition}
\newtheorem{remark}[thm]{Remark}
\newtheorem{prop}[thm]{Proposition}

\usepackage{amsmath, amssymb, amsthm,color,tikz, mathdots}
\usetikzlibrary{arrows,decorations.pathmorphing,backgrounds,fit}

\title[The determinant of the second additive compound of a matrix]{The determinant of the second additive compound of a square matrix: a formula and applications}

\author {Murad Banaji}

\begin{document}

\begin{abstract}
A formula is presented for the determinant of the second additive compound of a square matrix in terms of coefficients of its characteristic polynomial. This formula can be used to make claims about the eigenvalues of polynomial matrices, with sign patterns as an important special case. A number of corollaries and applications of this formula are given.

\vspace{0.5cm}

\end{abstract}


\maketitle

\section{Introduction and motivation}

\subsection{A general problem} Let $X = (X_1, \ldots, X_k)$ and let $M\in \mathbb{R}[X]^{n \times n}$, namely, $M$ is an $n \times n$ matrix whose entries are real polynomials in $X_1, \ldots, X_k$. Let $\mathcal{X} \subseteq \mathbb{R}^k$ be a semialgebraic set: a subset of $\mathbb{R}^k$ defined by polynomial equations and inequalities. Problems involving the eigenvalues of $M(x)$ as $x$ varies over $\mathcal{X}$ arise naturally in the study of differential equation models with network structure such as chemical reaction networks (CRNs), gene networks, or systems of interacting populations. Consider, for example, the claim:
\begin{quote}
(SNS2)\hspace{0.5cm} $M(x)$ has no pair of eigenvalues which sum to zero for any $x \in \mathcal{X}$.
\end{quote}
A certificate for SNS2 ensures, in particular, that $M(x)$ has no nonzero imaginary eigenvalues, ruling out so-called Hopf bifurcations \cite{Wiggins}. Indeed, it was questions about necessary conditions for Hopf bifurcation which motivated this work (see \cite{Guckenheimer1997aa, abphopf} for example). As is well known, SNS2 is a problem about polynomial positivity. In order to clarify this, first we fix some terminology.  

Let $p \in\mathbb{R}[X]$. If $p>0$ or $p<0$ on $\mathcal{X} \subseteq \mathbb{R}^k$, we say that $p$ is {\em sign-definite} on $\mathcal{X}$; if $p \geq 0$ or $p \leq 0$ on $\mathcal{X}$, it is {\em sign-semidefinite} on $\mathcal{X}$; otherwise, $p$ is {\em sign-indefinite} on $\mathcal{X}$ and we denote this by $p \gtrless 0$ on $\mathcal{X}$. In this paper, unless stated otherwise, $\mathcal{X}$ is taken to be the positive orthant, namely $\mathbb{R}^k_{\gg 0} := \{x \in \mathbb{R}^k\colon x_i > 0\,\,\mbox{for}\,\,i=1, \ldots, k\}$. So, for example, writing ``$p >0$'' or that ``$p$ is positive'' without further qualification means that it is positive for positive values of its variables (and not necessarily globally positive). We refer to a polynomial with both positive and negative terms as having {\em mixed terms}. Note that if $p\gtrless0$ on the positive orthant, then $p$ has mixed terms, but the converse does not hold. 

SNS2 is a claim about polynomial positivity because $M$ satisfies SNS2 if and only if $\mathrm{det}\,M^{[2]}$, the determinant of the second additive compound of $M$ (defined below), is sign-definite on $\mathcal{X}$. Given a particular $M$, we might directly examine $\mathrm{det}\,M^{[2]}$, and hope to use results such as the Positivstellensatz of Krivine and Stengle (see \cite{BCR,Parrilo03}, for example) to show that $\mathrm{det}\,M^{[2]}$ is sign-definite on $\mathcal{X}$. 

However, even in relatively low dimensions (e.g., $n=5$), and with entries in $M$ being linear forms, we encounter cases where $\mathrm{det}\,M^{[2]}$ is sign-definite or sign-semidefinite but showing this directly is quite nontrivial (some examples are presented later). One approach is to consider polynomial changes of variables. In particular, letting $p:=\mathrm{det}\,M^{[2]}$, can we find polynomials $Y:=(Y_1(X), \ldots, Y_m(X))$, and a new polynomial function $q$ defined via $p(X) = q(Y(X))$, such that the problem of determining the sign of $q$ on $Y(\mathcal{X})$ is easier than the original problem of determining the sign of $p$ on $\mathcal{X}$? The best choice for $Y$ may depend on the problem details, but one canonical choice is to set $Y_i$ ($i=1,\ldots, n$) to be the sum of the $i \times i$ principal minors of $M$, termed the $i$th {\em minor-sum} of $M$. Equivalently, the $Y_i$ are, upto sign, the coefficients of the characteristic polynomial of $M$. That such a change of variables is possible follows from basic arguments about symmetric polynomials, and writing down a general formula for $\mathrm{det}\,M^{[2]}$ in terms of these new variables is straightforward (Section~\ref{secadcomp}). This formula, already implicit in \cite{Guckenheimer1997aa}, proves particularly useful when studying sign patterns, which are now introduced.

\subsection{Sign patterns} An $n \times n$ sign pattern \cite{brualdi}, or an ``$n$-pattern'' for short, can be regarded, equivalently, as an $n \times n$ polynomial matrix in positive variables, a (convex) set of real $n \times n$ matrices, or an edge-weighted digraph on $n$ vertices. For example, given
\[
A=\left(\begin{array}{rr}1&-1\\1&0\end{array}\right),
\]
the sign pattern of $A$ is associated with the three objects:
\begin{enumerate}
\item The polynomial matrix $P_A:=\displaystyle{\left(\begin{array}{cc}X_1&-X_2\\X_3&0\end{array}\right)}$ with $X_1, X_2, X_3$ assumed real and positive.
\item The set of matrices $\mathcal{Q}_A:=\left\{B \in \mathbb{R}^{2 \times 2}\colon B_{11}>0, B_{12}<0, B_{21}>0, B_{22}=0\right\}$ termed the {\em qualitative class} of $A$. 
\item The {\em signed digraph} $G_A:=$ 
\begin{tikzpicture}[scale=1.2, baseline=-1mm]
\fill[color=black] (0,0) circle (1.5pt);
\fill[color=black] (1,0) circle (1.5pt);

\draw [-, thick] (-0.05,0.05) .. controls (-0.25,0.25) and (-0.4,0.1) .. (-0.4,0);
\draw [-, thick] (-0.05,-0.05) .. controls (-0.25,-0.25) and (-0.4,-0.1) .. (-0.4,0);
\draw [->, dashed, thick] (0.1,0.05) .. controls (0.4,0.15) and (0.6,0.15) .. (0.9,0.05);
\draw [->, thick] (0.9,-0.05) .. controls (0.6,-0.15) and (0.4,-0.15) .. (0.1,-0.05);
\node at (0,-0.2) {$\scriptstyle{1}$};
\node at (1,-0.2) {$\scriptstyle{2}$};
\end{tikzpicture} Each qualitative class $\mathcal{Q}_A$ clearly includes a unique $(-1,0,1)$ matrix, say $A'$, and $G_A$ is the signed digraph with adjacency matrix $A'$. We draw positive arcs with continuous lines and negative arcs with dashed lines. 
\end{enumerate}

The default meaning of ``sign pattern'' here is the first one: a polynomial matrix with positive variables. Given either a fixed matrix or a sign pattern $M$, the associated qualitative class and signed digraph will be denoted $\mathcal{Q}_M$ and $G_M$ respectively. Consider the following conditions on a real square matrix $A$:
\begin{itemize}
\item[(i)] $\mathrm{det}\,P_A$ is sign-definite. A matrix satisfying this condition is termed ``sign-nonsingular'', and characterising sign-nonsingular matrices is the classic sign-nonsigularity problem (SNS). 
\item[(ii)] $\mathrm{det}\,P_A^{[2]}$ is sign-definite. This is just SNS2 specialised to a sign pattern. 
\end{itemize}
This superficial similarity between (i) and (ii) is worthy of further comment. SNS is {\em algebraically} uninteresting in the following sense: $\mathrm{det}\,P_A$ is either sign-definite, sign-indefinite or identically zero, being sign-indefinite if and only if it has mixed terms. These facts are well-known and follow from easy observations about Newton Polytopes (see Lemma~\ref{lemsgnindef} and Remark~\ref{remSNS}). However, SNS is still {\em combinatorially} interesting: it is natural to search for a combinatorial characterisation of SNS matrices in terms of their signed digraphs or related structures. There are a variety of interesting results in this direction (e.g., \cite{Thomassen1,SeymourThomassen,vazirani}), culminating in the surprising graph-theoretic result that deciding if a matrix is SNS is in {\bf P} \cite{RobertsonSeymourThomas}.

Attempting any similar combinatorial work on $\mathrm{det}\,P_A^{[2]}$, we face the immediate hurdle that problems involving $\mathrm{det}\,P_A^{[2]}$ are {\em not} algebraically trivial. If $A$ is an $n \times n$ matrix with $n \geq 4$, then $\mathrm{det}\,P_A^{[2]}$ may be nonzero and sign-semidefinite but not sign-definite, or may be sign-definite while having mixed terms (Examples~\ref{exnz}~and~\ref{exmixed} below). Worse still, it is easy to find examples with $n=5$ where $-\mathrm{det}\,P_A^{[2]} \geq 0$, but $-\mathrm{det}\,P_A^{[2]}$ does not belong to the smallest cone in $\mathbb{R}[X]$ containing the monomials $\{X_1, \ldots, X_k\}$ (see Example~\ref{exharder} and the discussion in Section~\ref{secpsatz} below for the meaning and significance of this fact). Thus even preliminiary computational exploration of the sign of $\mathrm{det}\,P_A^{[2]}$ requires some nontrivial machinery involving positivity of polynomials.

\section{The second additive compound of a square matrix}
\label{secadcomp}

In this section, we briefly describe additive compounds and present the basic formula for the determinant of a second additive compound which will be used subsequently.

Given an $n$-dimensional vector space $V$ over some field $K$, let $\Lambda^2 V$ be the second exterior power of $V$, which can be identified with the ${n \choose 2}$-dimensional $K$-vector space consisting of the antisymmetric elements of $V \otimes V$. $\Lambda^2V$ is the set of finite linear combinations of elements of the form $u_1 \wedge u_2$ where $u_i \in V$, and ``$\wedge$'' is the anticommutative or ``wedge''-product, namely an associative, distributive, product satisfying $u \wedge v = -v \wedge u$ and $\lambda u \wedge v = u \wedge \lambda v = \lambda(u \wedge v)$ for all $u,v \in V$ and scalars $\lambda$. The elements of $\Lambda^2 V$ are termed {\em bivectors}. 

Consider a linear map $L\colon V \to V$ with eigenvalues $\{\lambda_1, \lambda_2, \ldots, \lambda_n\}$. $L$ induces linear maps on $\Lambda^2 V$ in two important ways:
\begin{enumerate}
\item The second exterior power (or multiplicative compound) of $L$, denoted $L^{(2)}$ is defined via $L^{(2)}(u \wedge v) = Lu \wedge Lv$, and extends to all bivectors by linearity. Its eigenvalues are precisely the products of pairs of eigenvalues of $L$, namely the multiset $\{\lambda_i\lambda_j\,|\, i < j\}$. 
\item The second additive compound of $L$, denoted $L^{[2]}$ is defined via $L^{[2]}(u \wedge v) = Lu \wedge v + v \wedge Lv$, and again extends to all bivectors by linearity. Its eigenvalues are precisely the {\em sums} of pairs of eigenvalues of $L$ counted with multiplicity, namely the multiset $\{\lambda_i+\lambda_j\,|\, i < j\}$. 
\end{enumerate}
Higher multiplicative and additive compounds can also be defined naturally. The reader is referred to \cite{allenbridges} for an introduction which focusses on applications to differential equations. Here our interest is in the second additive compound $L^{[2]}$, precisely because of its spectral properties.  

Any basis $\mathcal{B}$ of $V$ naturally induces a basis $\mathcal{B}'$ on $\Lambda^2 V$ consisting of wedge products of distinct pairs of vectors in $\mathcal{B}$. A fixed ordering on the elements of $\mathcal{B}$ can be used to fix an ordering on the elements of $\mathcal{B}'$ (the most common choice being the lexicographic ordering), and so a matrix representation of $L$, say $M$, gives rise to a corresponding matrix representation of $L^{[2]}$, say $M^{[2]}$, the {\em second additive compound matrix} of $M$. The nonzero entries of $M^{[2]}$ are simple linear forms in the entries of $M$; an explicit formula is given in \cite{muldowney}. Since our interest here is solely in $\mathrm{det}\,M^{[2]}$, $\mathcal{B}$ is arbitrary.

Define $\mathrm{det}^{[2]}$ to be the map which takes a square matrix to the determinant of its second additive compound, namely $\mathrm{det}^{[2]}\,M = \mathrm{det}(M^{[2]})$. Let $n \geq 2$ and let $\mathbb{C}^{n \times n}$ denote the $n \times n$ complex matrices. Define $\mathrm{det}^{[2]}_n \colon \mathbb{C}^{n \times n} \to \mathbb{C}$ to be the restriction of $\mathrm{det}^{[2]}$ to $\mathbb{C}^{n \times n}$. The spectrum of $M \in \mathbb{C}^{n \times n}$ can be regarded as 
a point in the quotient space $\mathbb{C}^n/S_n$ where $S_n$ is the symmetric group with the natural action on $\mathbb{C}^n$. Functions on $\mathbb{C}^n/S_n$ are just symmetric functions on $\mathbb{C}$, and in particular, polynomial functions on $\mathbb{C}^n/S_n$ are those represented by symmetric polynomials in $n$ indeterminates. For each $n \geq 2$ define the maps:
\begin{enumerate}
\item $\mathrm{spec}_n \colon \mathbb{C}^{n \times n} \to \mathbb{C}^n/S_n$ by $\mathrm{spec}_n(M) = (\lambda_1, \ldots, \lambda_n)$ where $\lambda_1, \ldots, \lambda_n$ are the eigenvalues of $M$. Namely, $\mathrm{spec}_n(\cdot)$ takes a matrix to its spectrum.
\item $\mathrm{char}_n\colon \mathbb{C}^{n \times n} \to \mathbb{C}^n$ by $\mathrm{char}_n(M) = (J_1, \ldots, J_n)$, where $J_k$ is the $k$th minor-sum of $M$. As the $J_i$ are, upto sign, the non-leading coefficients of the characteristic polynomial of $M$, we can think of $\mathrm{char}_n(\cdot)$ as taking a matrix to its characteristic polynomial. 
\item $e_n\colon \mathbb{C}^n/S_n \to \mathbb{C}^n$ by $e_n(\lambda_1, \ldots, \lambda_n) = (J_1, \ldots, J_n)$, where $J_k$ is the elementary symmetric polynomial of degree $k$ in $\lambda_1, \ldots, \lambda_n$. Namely, $e_n$ can be regarded as the function taking the spectrum of a matrix to the coefficients of its characteristic polynomial.
\item $p_n\colon \mathbb{C}^n/S_n \to \mathbb{C}$ by $p_n(\lambda_1, \ldots, \lambda_n)=\prod_{i < j}(\lambda_i+\lambda_j)$. Namely, $p_n(\cdot)$ takes the spectrum of a matrix to the determinant of its second additive compound. (Note that $p_n$ is a symmetric polynomial function, so this makes sense.)
\end{enumerate}
Consider the following diagram:

\begin{center}
\begin{tikzpicture}[domain=0:4,scale=0.8]

\node at (-2,4) {$\mathcal{X}$};
\node at (1,4) {$\mathbb{C}^{n \times n}$};
\node at (4,4) {$\mathbb{C}^n/S_n$};
\node at (1,1) {$\mathbb{C}^n$};
\node at (4,1) {$\mathbb{C}$};

\draw[->, thick] (1.6,4.0) -- (3.2, 4.0);
\draw[->, thick] (-1.6,4.0) -- (0.2, 4.0);
\draw[->, thick] (1.4,1.0) -- (3.6, 1.0);

\draw[->, thick] (3.6,3.6) -- (1.4, 1.4);

\draw[->, thick] (1.0, 3.6) -- (1.0, 1.4);
\draw[->, thick] (4.0, 3.6) -- (4.0, 1.4);

\node at (0.4, 2.5) {$\mathrm{char}_n$};
\node at (2.5, 4.4) {$\mathrm{spec}_n$};
\node at (-0.5, 4.4) {$M$};

\node at (4.4, 2.5) {$p_n$};
\node at (2.5, 0.6) {$q_n$};

\node at (2.3, 2.7) {$e_n$};
\end{tikzpicture}
\end{center} 
$\mathcal{X}$ is some set and the arrow from $\mathcal{X}$ assigns to each $x \in \mathcal{X}$ an $n \times n$ complex matrix, say $M(x)$. The functions $p_n$, $\mathrm{spec}_n$, $\mathrm{char}_n$ and $e_n$ are described above. Commutativity of the upper left triangle is of course well-known, and allows us to abuse notation by writing $J_i(x)$ when referring to $(\mathrm{char}_n\circ M)_i(x)$, or $J_i(\lambda_1, \ldots, \lambda_n)$ when referring to $(e_n(\lambda_1, \ldots, \lambda_n))_i$, depending on context. 

The goal is to write down a polynomial function $q_n$ which makes the above diagram commute, namely to find $q_n\colon \mathbb{C}^n \to \mathbb{C}$ such that $\mathrm{det}^{[2]}_n = q_n\circ \mathrm{char}_n$. Existence of $q_n$ follows from the Fundamental Theorem of Symmetric Polynomials (\cite{blumsmith}, for example), which tells us that there is a unique, polynomial function $q_n$ which makes the lower right triangle (and hence the whole diagram) commute. Thus, for each $n \geq 2$, we define $q_n$ via $q_n\circ e_n = p_n$, or more explicitly:
\[
q_n(J_1(\lambda_1, \ldots, \lambda_n), J_2(\lambda_1, \ldots, \lambda_n), \ldots, J_n(\lambda_1, \ldots, \lambda_n)) = p_n(\lambda_1, \ldots, \lambda_n)
\]
For brevity, let $J = (J_1, \ldots, J_n)$ and $\lambda = (\lambda_1, \ldots, \lambda_n)$, in which case, the defining equation of $q_n$ becomes tidier:
\begin{equation}
\label{qneq}
q_n(J(\lambda)) = p_n(\lambda)\,.
\end{equation}
\begin{lemma}
\label{lemirred}
$q_n$, as defined by (\ref{qneq}), has degree $n-1$, and is irreducible as an element of $\mathbb{C}[J]$.
\end{lemma}
\begin{proof}
Recall that $p_n = \prod_{1\leq i < j\leq n}(\lambda_i + \lambda_j)$. The monomial $m:=\lambda_1^{n-1}\lambda_2^{n-2}\cdots \lambda_{n-1}^1$ occurs (with coefficient $1$) in $p_n$, while clearly $\lambda_1^n$ does not divide any monomial of $p_n$. Each monomial of $q_n$ is a product $\prod_{i=1}^rs_i$, where each $s_i$ is a nonconstant elementary symmetric polynomial in $\lambda$. But: (i) $\lambda_1^n$ does not divide any monomial of $p_n$ and so $q_n$ has degree $\leq n-1$; (ii) $m$ occurs in $p_n$ and so $q_n$ must have degree $\geq n-1$. Thus $q_n$ must have degree $n-1$.

Let $q_n = q'q''$, so that $p_n = p'p''$, where $p'(\lambda):=q'(J(\lambda))$ and $p''(\lambda):=q''(J(\lambda))$. Since the irreducible polynomial $\lambda_1+\lambda_2$ divides $p_n$, it divides one of $p'$ or $p''$; w.l.o.g. let this be $p'$. But then, as $p'$ is symmetric, $\lambda_i + \lambda_j$ divides $p'$ for every pair $(i,j)$, $1\leq i < j \leq n$. Consequently, the product of these factors, namely $p_n$ itself, divides $p'$. Thus $p''=q''$ is a nonzero constant, and hence $q_n$ is irreducible over $\mathbb{C}$.
\end{proof}

\begin{prop} 
\label{J2formula}
Let $M \in \mathbb{R}[X]^{n \times n}$ ($n \geq 2$), and let $J_i \in \mathbb{R}[X]$ ($i = 1, \ldots, n$), be the $i$th minor-sum of $M$. Define $J_i := 0$ for $i > n$. Then $q_n(J_1, \ldots, J_n) = \mathrm{det}^{[2]}M$ is the leading $(n-1) \times (n-1)$ principal minor of 
\begin{equation}
\label{J2eqn}
\mathcal{M} := \left(\begin{array}{ccccc}J_1&1&0&0&\hdots\\J_3&J_2&J_1&1&\hdots\\J_5&J_4&J_3&J_2&\hdots\\J_7&J_6&J_5&J_4&\hdots\\\vdots&\vdots&\vdots&\vdots&\ddots\end{array}\right)\,.
\end{equation}

\end{prop}

\begin{remark}
The proof of Proposition~\ref{J2formula} is essentially contained in \cite{Guckenheimer1997aa}, but is presented for completeness. Here is $q_n$ for $n = 2,3,4,5$:
\[
\begin{array}{rcl}
q_2&=& J_1,\\
q_3&=& -J_3 + J_1J_2 = \left|\begin{array}{cc}J_1&1\\J_3&J_2\end{array}\right|\,,\\
q_4&=& -J_1^2J_4 - J_3^2 + J_1J_2J_3 = \left|\begin{array}{ccc}J_1&1&0\\J_3&J_2&J_1\\0&J_4&J_3\end{array}\right|\,,\\
q_5&=& -J_5^2 + 2J_1J_4J_5+J_2J_3J_5-J_1J_2^2J_5 - J_1^2J_4^2-J_3^2J_4 + J_1J_2J_3J_4 = \left|\begin{array}{cccc}J_1&1&0&0\\J_3&J_2&J_1&1\\J_5&J_4&J_3&J_2\\0&0&J_5&J_4\end{array}\right|\,.
\end{array}
\]
\end{remark}

\begin{proof}[Proof of Proposition~\ref{J2formula}] Fix $n$ and refer to the leading $r \times r$ principal submatrix of $\mathcal{M}$ as $\mathcal{M}_{r}$, so our claim is that $q_n = \mathrm{det}\,\mathcal{M}_{n-1}$. By observation, $\mathrm{det}\,\mathcal{M}_{n-1}$ is a polynomial of degree $n-1$ in $J$. 

Define $J_0 = 1$, $J_k=0$ for $k<0$, and define the polynomials
\[
\hat{P}(J,\mu) := J_n +\mu^2 J_{n-2} + \mu^4 J_{n-4} + \cdots\,, \quad \hat{Q}(J,\mu) := J_{n-1} + \mu^2 J_{n-3} + \cdots\,.
\]
It is straightforward to confirm that 
\[
\frac{1}{2}\left(\mathrm{det}\,(M+\mu I) + \mathrm{det}\,(M-\mu I)\right) = \hat{P}(J,\mu), \quad \frac{1}{2}\left(\mathrm{det}\,(M+\mu I) -  \mathrm{det}\,(M-\mu I)\right) = \mu\hat{Q}(J,\mu)\,.
\]
We now regard $\hat{P}, \hat{Q}$ as elements of $\mathbb{Z}[J][\mu]$ and claim that they have a common root $\mu_0 \in \mathbb{C}$ at some $x \in \mathbb{C}^k$ (i.e., there exists $\mu_0 \in \mathbb{C}$ such that $\hat{P}(J(x), \mu_0)=\hat{Q}(J(x), \mu_0) = 0$) if and only if $M(x)$ has a pair of eigenvalues which sum to zero, namely, either (i) $M(x)$ has a pair of opposite eigenvalues $\pm \mu_0$ ($\mu_0 \neq 0$), or (ii) $M(x)$ has a zero eigenvalue of multiplicity at least two. To see this, note that $\hat{P}(J,0) = \hat{Q}(J,0)=0$ if and only if $J_n = J_{n-1} = 0$, namely $M$ has a zero eigenvalue of multiplicity at least two; on the other hand, if $\mu_0 \neq 0$, then $\hat{P}(J,\mu_0)=\hat{Q}(J,\mu_0)=0$ if and only if $\mathrm{det}\,(M+\mu_0 I) = \hat{P}(J,\mu_0) + \mu_0\hat{Q}(J,\mu_0) = 0$ and $\mathrm{det}\,(M-\mu_0 I) = \hat{P}(J,\mu_0) - \mu_0\hat{Q}(J,\mu_0) = 0$, namely $M$ has eigenvalues $\pm \mu_0$. 

By definition of the resultant (Chapter~4 of \cite{Sturmfels02solvingsystems}), $\hat{P}$ and $\hat{Q}$ have a common root (in $\mathbb{C}$) if and only if the resultant $\mathrm{Res}_{\mu}(\hat{P},\hat{Q}) = 0$. But $\mathcal{M}_{n-1}$ is precisely the Sylvester matrix of $\hat{P}, \hat{Q}$ after some row and column permutations, and so $\mathrm{det}\,\mathcal{M}_{n-1}$ is, upto sign, equal to $\mathrm{Res}_{\mu}(\hat{P},\hat{Q})$. Thus $\mathrm{det}\,\mathcal{M}_{n-1}$ has the same zero set as $q_n$. But (Lemma~\ref{lemirred}) $q_n$ is irreducible and both $q_n$ and $\mathrm{det}\,\mathcal{M}_{n-1}$ have degree $n-1$. Thus, by Hilbert's Nullstellensatz, they are equal upto some constant, which we now show to be one.

We have already observed in the proof of Lemma~\ref{lemirred} that the monomial $\lambda_1^{n-1}\lambda_2^{n-2}\cdots \lambda_{n-1}^1$ occurs in $p_n$ with coefficient $1$. We claim that the unique way that this monomial arises in $\mathrm{det}\,\mathcal{M}_{n-1}$ (regarded as a polynomial in $\lambda$) is from the product $J_1J_2\cdots J_{n-1}$, where it occurs with coefficient $1$. This is easy by induction on $n$. The base case $n=2$ is trivial as $\mathrm{det}\,\mathcal{M}_{1} = J_1 = \lambda_1+\lambda_2$. Suppose the claim holds for $n=k-1$. By observation, $\mathrm{det}\,\mathcal{M}_{k-1} = J_{k-1}\mathrm{det}\,\mathcal{M}_{k-2} + J_{k}R$ where $R$ is some polynomial. Since the monomial $\lambda_1^{k-1}\lambda_2^{k-2}\cdots \lambda_{k-1}^1$ does not include $\lambda_k$ it clearly cannot occur in $J_{k}R$ (regarded as a polynomial in $\lambda$), and so occurs only in $J_{k-1}\mathrm{det}\,\mathcal{M}_{k-2}$, and with the same coefficient as $\lambda_1^{k-2}\lambda_2^{k-3}\cdots \lambda_{k-2}^1$ occurs in $\mathrm{det}\,\mathcal{M}_{k-2}$ namely, by the inductive hypothesis, $1$. Thus, writing both $q_n$ and $\mathrm{det}\,\mathcal{M}_{n-1}$ as polynomials in $\lambda$, the monomial $\lambda_1^{n-1}\lambda_2^{n-2}\cdots \lambda_{n-1}^1$ occurs with the same coefficient in both. This completes the proof that $q_n = \mathrm{det}\,\mathcal{M}_{n-1}$.
\end{proof}

\section{Some examples}

Rather than moving straight on to further results, we first present some examples. These demonstrate how Proposition~\ref{J2formula}, and some additional corollaries and related results to be proved later, can be used to make claims about $\mathrm{det}^{[2]}\,M$ for various real polynomial matrices $M$. If $M$ is an $n$-pattern, then $G_M$ refers to the associated signed digraph, and $G_{\overline{M}}$ refers to the associated unsigned digraph (namely, $G_M$ without edge-signs). In each case, $J_i$ refers to $i$th minor-sum of $M$. A homogeneous polynomial of degree $k$ is termed a ``$k$-form''. 

\begin{example}
Consider the following $3$-pattern $M$, and associated digraph:

\begin{center}
\begin{tikzpicture}[scale=1.5]

\node at (-5,-0.4) {$M=\left(\begin{array}{ccc}X_1&-X_2&-X_3\\X_4&X_5&X_6\\X_7&0&X_8\end{array}\right)$};

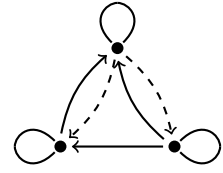
\fill[color=black] (0,0) circle (1.5pt);
\fill[color=black] (-0.5,-0.87) circle (1.5pt);
\fill[color=black] (0.5,-0.87) circle (1.5pt);

\begin{scope}[rotate=-60]
\draw [->, thick, dashed] (0.1,0.03) .. controls (0.4,0.13) and (0.6,0.13) .. (0.9,0.03);
\draw [->, thick] (0.9,-0.05) .. controls (0.6,-0.15) and (0.4,-0.15) .. (0.1,-0.05);
\end{scope}

\begin{scope}[rotate=-120]
\draw [->, thick, dashed] (0.1,0.03) .. controls (0.4,0.13) and (0.6,0.13) .. (0.9,0.03);
\draw [->, thick] (0.9,-0.05) .. controls (0.6,-0.15) and (0.4,-0.15) .. (0.1,-0.05);
\end{scope}

\draw[<-, thick] (-0.4,-0.87) -- (0.4,-0.87); 

\draw [-, thick] (0.05,0.05) .. controls (0.25,0.25) and (0.1,0.4) .. (0,0.4);
\draw [-, thick] (-0.05,0.05) .. controls (-0.25,0.25) and (-0.1,0.4) .. (0,0.4);

\begin{scope}[xshift=0.5cm,yshift=-0.87cm]
\draw [-, thick] (0.05,0.05) .. controls (0.25,0.25) and (0.4,0.1) .. (0.4,0);
\draw [-, thick] (0.05,-0.05) .. controls (0.25,-0.25) and (0.4,-0.1) .. (0.4,0);
\end{scope}

\begin{scope}[xshift=-0.5cm,yshift=-0.87cm]
\draw [-, thick] (-0.05,0.05) .. controls (-0.25,0.25) and (-0.4,0.1) .. (-0.4,0);
\draw [-, thick] (-0.05,-0.05) .. controls (-0.25,-0.25) and (-0.4,-0.1) .. (-0.4,0);
\end{scope}
\end{tikzpicture}
\end{center}
Without computing $\mathrm{det}^{[2]}\,M$ (which is, of course, easy enough) we can conclude from examination of $G_M$ that $\mathrm{det}^{[2]}\,M>0$. This follows from Proposition~\ref{prop33} below. 

\end{example}

\begin{example}
\label{ex4c}
Consider the following $4$-pattern $M$ and associated digraph:

\begin{center}
\begin{tikzpicture}[scale=1.5]
\begin{scope}[rotate=0]
\fill[color=black] (0,0) circle (1.5pt);
\fill[color=black] (-0.5,-0.87) circle (1.5pt);
\fill[color=black] (0.5,-0.87) circle (1.5pt);
\fill[color=black] (0,0.5) circle (1.5pt);

\draw[->,thick, dashed] (0,0.42) -- (0,0.08);

\begin{scope}[xshift=-0.5cm,yshift=-0.87cm,rotate=70]
\draw [->, thick] (0.1,0.07) .. controls (0.57,0.17) and (0.87,0.17) .. (1.34,0.07);
\end{scope}

\begin{scope}[xshift=0.5cm,yshift=-0.87cm,rotate=110]
\draw [->, thick,dashed] (1.34,-0.07) .. controls (0.87,-0.17) and (0.57,-0.17) .. (0.1,-0.07);
\end{scope}

\begin{scope}[rotate=-60]
\draw [->, thick,dashed] (0.1,0.03) .. controls (0.4,0.13) and (0.6,0.13) .. (0.9,0.03);
\draw [->, thick,dashed] (0.9,-0.05) .. controls (0.6,-0.15) and (0.4,-0.15) .. (0.1,-0.05);
\end{scope}
\begin{scope}[rotate=-120]
\draw [->, thick,dashed] (0.1,0.03) .. controls (0.4,0.13) and (0.6,0.13) .. (0.9,0.03);
\draw [->, thick] (0.9,-0.05) .. controls (0.6,-0.15) and (0.4,-0.15) .. (0.1,-0.05);
\end{scope}
\begin{scope}[xshift=-0.5cm,yshift=-0.87cm]
\draw [->, thick] (0.1,0.03) .. controls (0.4,0.13) and (0.6,0.13) .. (0.9,0.03);
\draw [->, thick,dashed] (0.9,-0.05) .. controls (0.6,-0.15) and (0.4,-0.15) .. (0.1,-0.05);
\end{scope}

\begin{scope}[xshift=-0.5cm,yshift=-0.87cm]
\draw [-, thick,dashed] (-0.05,0.05) .. controls (-0.25,0.25) and (-0.4,0.1) .. (-0.4,0);
\draw [-, thick,dashed] (-0.05,-0.05) .. controls (-0.25,-0.25) and (-0.4,-0.1) .. (-0.4,0);
\end{scope}
\end{scope}

\node at (-5,-0.3) {$M=\left(\begin{array}{cccc}-X_1&X_2&X_3&X_4\\0&0&-X_5&-X_6\\-X_7&0&0&-X_8\\-X_{9}&0&-X_{10}&0\end{array}\right)$};

\end{tikzpicture}
\end{center}
$J_3>0$, $J_4>0$ and $J_3-J_1J_2>0$, and consequently, from the formula in Proposition~\ref{J2formula}, $q_4 = -J_1^2J_4 - J_3(J_3 - J_1J_2)<0$. The signs of $J_3, J_4$ and $J_3-J_1J_2$ are easily computed, but are also verifiable by examining cycles in $G_M$ (see the discussion in Section~\ref{secsignpat} and in Remark~\ref{min3mixed} below). 
\end{example}

\begin{example}
\label{ex4a}
Below is a $4\times 4$ matrix of $1$-forms. $-M$ arises as the Jacobian matrix of the CRN shown to the right under weak assumptions on the rates of reaction.
\begin{center}
\begin{tikzpicture}[scale=1.6]


\node at (-4.5,-0.42) {$M=\left(\begin{array}{ccccc}X_1&X_2-X_3&0&-X_4\\X_1&X_2+X_5&-X_6-X_7&-X_8\\-X_1&-X_2-X_5&X_6+X_7&X_8\\0&X_3-X_5&-X_6+X_7&X_4+X_8\end{array}\right)$};

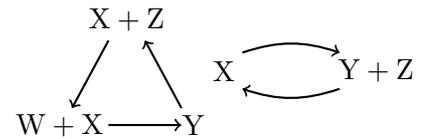
\node at (-0.05,0) {$\mathrm{X}+\mathrm{Z}$};
\node at (-0.6,-0.87) {$\mathrm{W}+\mathrm{X}$};
\node at (0.5,-0.87) {$\mathrm{Y}$};

\draw[->, thick] (-0.2,-0.17) -- (-0.5,-0.73); 
\draw[<-, thick] (0.1,-0.17) -- (0.4,-0.73); 
\draw[->, thick] (-0.2,-0.87) -- (0.4,-0.87);

\begin{scope}[xshift=0.8cm,yshift=-0.42cm]
\node at (-0.05,0) {$\mathrm{X}$};
\node at (1.2,0) {$\mathrm{Y+Z}$};
\draw [->, thick] (0.1,0.15) .. controls (0.4,0.25) and (0.6,0.25) .. (0.9,0.15);
\draw [->, thick] (0.9,-0.15) .. controls (0.6,-0.25) and (0.4,-0.25) .. (0.1,-0.15);
\end{scope}

\end{tikzpicture}
\end{center}
$X_1, \ldots, X_8$ are real, positive variables. Unlike the examples above, and most of those to follow, $M$ is not a sign pattern and the corresponding set of matrices $\{M(x_1, \ldots, x_8) \colon x_1, \ldots, x_8 >0\}$ is not a subset of any qualitative class. $\mathrm{det}^{[2]}\,M$ is a $6$-form in $X_1, \ldots, X_8$ with 203 terms. However, $\mathrm{det}\,M$ vanishes identically, namely $J_4=0$. This follows, without any calculation, from a little theory of CRNs \cite{gunawardenaCRNT}: the fact that there are $5$ reactions (arcs) and two connected components in the digraph above implies a factorisation of $M$ where the first factor (the so-called ``stoichiometric matrix'') has rank $\leq 3$. Thus, according to the formula in Proposition~\ref{J2formula}, $\mathrm{det}^{[2]}M$ factorises as $J_3(J_1J_2-J_3)$. Each factor is easily computed to be positive on $\mathrm{R}^8_{\gg 0}$, and so $\mathrm{det}^{[2]}\,M>0$ on $\mathrm{R}^8_{\gg 0}$ without full calculation of $\mathrm{det}^{[2]}\,M$. In particular, this CRN is incapable of Hopf bifurcations. 
\end{example}

\begin{example}
\label{ex4b}
Consider the following $4$-pattern and associated digraph

\begin{center}
\begin{tikzpicture}[scale=1.5]

\fill[color=black] (0,0) circle (1.5pt);
\fill[color=black] (-0.87,-0.5) circle (1.5pt);
\fill[color=black] (0.87,-0.5) circle (1.5pt);
\fill[color=black] (0,-1) circle (1.5pt);

\draw [-, thick] (0.05,0.05) .. controls (0.25,0.25) and (0.1,0.4) .. (0,0.4);
\draw [-, thick] (-0.05,0.05) .. controls (-0.25,0.25) and (-0.1,0.4) .. (0,0.4);

\begin{scope}[xshift=-0.87cm,yshift=-0.5cm]
\draw [-, thick] (-0.05,0.05) .. controls (-0.25,0.25) and (-0.4,0.1) .. (-0.4,0);
\draw [-, thick] (-0.05,-0.05) .. controls (-0.25,-0.25) and (-0.4,-0.1) .. (-0.4,0);
\end{scope}

\begin{scope}[yshift=-1cm]
\draw [-, thick] (0.05,-0.05) .. controls (0.25,-0.25) and (0.1,-0.4) .. (0,-0.4);
\draw [-, thick] (-0.05,-0.05) .. controls (-0.25,-0.25) and (-0.1,-0.4) .. (0,-0.4);
\end{scope}

\begin{scope}[xshift=0.87cm,yshift=-0.5cm]
\draw [-, thick] (0.05,0.05) .. controls (0.25,0.25) and (0.4,0.1) .. (0.4,0);
\draw [-, thick] (0.05,-0.05) .. controls (0.25,-0.25) and (0.4,-0.1) .. (0.4,0);
\end{scope}


\begin{scope}[rotate=210]
\draw [->, thick,dashed] (0.1,0.03) .. controls (0.4,0.13) and (0.6,0.13) .. (0.9,0.03);
\draw [->, thick] (0.9,-0.05) .. controls (0.6,-0.15) and (0.4,-0.15) .. (0.1,-0.05);
\end{scope}
\begin{scope}[rotate=-30]
\draw [->, thick] (0.1,0.03) .. controls (0.4,0.13) and (0.6,0.13) .. (0.9,0.03);
\draw [->, thick, dashed] (0.9,-0.05) .. controls (0.6,-0.15) and (0.4,-0.15) .. (0.1,-0.05);
\end{scope}

\begin{scope}[yshift=-1cm,rotate=30]
\draw [->, thick, dashed] (0.1,0.03) .. controls (0.4,0.13) and (0.6,0.13) .. (0.9,0.03);
\draw [->, thick] (0.9,-0.05) .. controls (0.6,-0.15) and (0.4,-0.15) .. (0.1,-0.05);
\end{scope}

\begin{scope}[yshift=-1cm,rotate=-210]
\draw [->, thick] (0.1,0.03) .. controls (0.4,0.13) and (0.6,0.13) .. (0.9,0.03);
\draw [->, thick, dashed] (0.9,-0.05) .. controls (0.6,-0.15) and (0.4,-0.15) .. (0.1,-0.05);
\end{scope}

\node at (-5,-0.5) {$M=\left(\begin{array}{cccc}X_1&X_2&0&-X_3\\-X_4&X_5&X_6&0\\0&-X_7&X_8&X_9\\X_{10}&0&-X_{11}&X_{12}\end{array}\right)$};

\end{tikzpicture}
\end{center}

As in the previous example, $M$ can (after a sign change) be regarded as the Jacobian matrix of a CRN,
but we may equally just consider it as a sign pattern. $\mathrm{det}^{[2]}\,M$ is a $6$-form in $X_1, \ldots, X_{12}$, with $194$ terms, of which $8$ are negative. Without computation, $\mathrm{det}^{[2]}\,M>0$ as $G_M$ satisfies a combinatorial sufficient condition for $[2]$-positivity of $4$-patterns, namely Proposition~\ref{prop44} below. In order to apply Proposition~\ref{prop44}, we observe the absence of $3$-cycles in $G_M$, along with the facts that all $1$- and $2$-cycles include an odd number of positive edges, while both $4$-cycles include an even number of positive edges.
\end{example}

\begin{example}
\label{exbasic}
Consider the following $5$-pattern and associated digraph:

\begin{center}
\begin{tikzpicture}[scale=1.5]

\fill[color=black] (-0.85,0) circle (1.5pt);
\fill[color=black] (-0.26,0.81) circle (1.5pt);
\fill[color=black] (-0.26,-0.81) circle (1.5pt);
\fill[color=black] (0.69,0.5) circle (1.5pt);
\fill[color=black] (0.69,-0.5) circle (1.5pt);

\begin{scope}[xshift=-0.85cm]
\draw [-, thick] (-0.05,0.05) .. controls (-0.25,0.25) and (-0.4,0.1) .. (-0.4,0);
\draw [-, thick] (-0.05,-0.05) .. controls (-0.25,-0.25) and (-0.4,-0.1) .. (-0.4,0);
\end{scope}


\begin{scope}[xshift=-0.88cm,rotate=-54]
\draw [->, thick] (0.1,0.03) .. controls (0.4,0.13) and (0.6,0.13) .. (0.9,0.03);
\draw [->, thick] (0.9,-0.05) .. controls (0.6,-0.15) and (0.4,-0.15) .. (0.1,-0.05);
\end{scope}
\begin{scope}[xshift=-0.85cm,rotate=54]
\draw [->, thick] (0.1,0.03) .. controls (0.4,0.13) and (0.6,0.13) .. (0.9,0.03);
\end{scope}

\begin{scope}[xshift=-0.26cm,yshift=0.81cm,rotate=-18]
\draw [->, thick,dashed] (0.1,0.03) .. controls (0.4,0.13) and (0.6,0.13) .. (0.9,0.03);
\end{scope}
\begin{scope}[xshift=-0.26cm,yshift=-0.81cm,rotate=18]
\draw [->, thick] (0.9,-0.05) .. controls (0.6,-0.15) and (0.4,-0.15) .. (0.1,-0.05);
\end{scope}
\begin{scope}[xshift=0.69cm,yshift=-0.5cm,rotate=90]
\draw [->, thick,dashed] (0.9,-0.05) .. controls (0.6,-0.15) and (0.4,-0.15) .. (0.1,-0.05);
\end{scope}

\draw[->,thick] (-0.77,0.025) -- (0.61,0.475);
\draw[->,thick] (-0.77,0) -- (0.61,-0.475);

\draw[->,thick,dashed] (-0.21,0.74) -- (0.64,-0.43);
\draw[->,thick] (-0.26,0.73) -- (-0.26,-0.73);

\draw[->,thick] (0.64,0.43) -- (-0.21,-0.74);

\node at (-5,0) {$M=\left(\begin{array}{ccccc}X_1&X_2&X_3&X_4&X_5\\0&0&-X_6&-X_7&X_8\\0&0&0&-X_9&X_{10}\\0&0&0&0&X_{11}\\X_{12}&0&0&0&0\end{array}\right)$};

\end{tikzpicture}
\end{center}
$\mathrm{det}^{[2]}\,M$ is a $10$-form in $X_1, \ldots, X_{12}$, with $38$ negative terms and $3$ positive terms. With a little effort we can complete some squares and confirm directly that $\mathrm{det}^{[2]}\,M<0$. However, this conclusion is obtained most easily by noting that $J_1, J_3, J_4, J_5>0$, while $J_2<0$, as verified by direct calculation, or examination of the cycles of $G_M$, as discussed in Section~\ref{secsignpat} below. Consequently, by Lemma~\ref{lem5pat1} below, $\mathrm{det}^{[2]}\,M<0$ on $\mathbb{R}^{12}_{\gg 0}$. 
\end{example}

\begin{example}
\label{exharder}
Consider the following $5$-pattern and associated digraph:
\begin{center}
\begin{tikzpicture}[scale=1.5]

\fill[color=black] (-0.85,0) circle (1.5pt);
\fill[color=black] (-0.26,0.81) circle (1.5pt);
\fill[color=black] (-0.26,-0.81) circle (1.5pt);
\fill[color=black] (0.69,0.5) circle (1.5pt);
\fill[color=black] (0.69,-0.5) circle (1.5pt);

\begin{scope}[xshift=0.69cm,yshift=0.5cm,rotate=198]
\draw [-, thick] (-0.05,0.05) .. controls (-0.25,0.25) and (-0.4,0.1) .. (-0.4,0);
\draw [-, thick] (-0.05,-0.05) .. controls (-0.25,-0.25) and (-0.4,-0.1) .. (-0.4,0);
\end{scope}
\begin{scope}[xshift=0.69cm,yshift=-0.5cm,rotate=162]
\draw [-, thick] (-0.05,0.05) .. controls (-0.25,0.25) and (-0.4,0.1) .. (-0.4,0);
\draw [-, thick] (-0.05,-0.05) .. controls (-0.25,-0.25) and (-0.4,-0.1) .. (-0.4,0);
\end{scope}


\begin{scope}[xshift=-0.88cm,rotate=-54]
\draw [->, thick] (0.9,-0.05) .. controls (0.6,-0.15) and (0.4,-0.15) .. (0.1,-0.05);
\end{scope}
\begin{scope}[xshift=-0.85cm,rotate=54]
\draw [->, thick] (0.1,0.03) .. controls (0.4,0.13) and (0.6,0.13) .. (0.9,0.03);
\end{scope}

\begin{scope}[xshift=-0.26cm,yshift=0.81cm,rotate=-18]
\draw [->, thick,dashed] (0.1,0.03) .. controls (0.4,0.13) and (0.6,0.13) .. (0.9,0.03);
\end{scope}
\begin{scope}[xshift=-0.26cm,yshift=-0.81cm,rotate=18]
\draw [->, thick] (0.9,-0.05) .. controls (0.6,-0.15) and (0.4,-0.15) .. (0.1,-0.05);
\end{scope}
\begin{scope}[xshift=0.69cm,yshift=-0.5cm,rotate=90]
\draw [->, thick,dashed] (0.9,-0.05) .. controls (0.6,-0.15) and (0.4,-0.15) .. (0.1,-0.05);
\end{scope}

\draw[->,thick] (-0.77,0.025) -- (0.61,0.475);
\draw[->,thick] (-0.77,0) -- (0.61,-0.475);

\draw[->,thick,dashed] (-0.21,0.74) -- (0.64,-0.43);
\draw[->,thick] (-0.26,0.73) -- (-0.26,-0.73);

\draw[->,thick] (0.64,0.43) -- (-0.21,-0.74);

\node at (-5,0) {$M=\left(\begin{array}{ccccc}0&X_1&X_2&X_3&0\\0&0&-X_4&-X_5&X_6\\0&0&X_7&-X_8&X_9\\0&0&0&X_{10}&X_{11}\\X_{12}&0&0&0&0\end{array}\right)$};

\end{tikzpicture}
\end{center}
$\mathrm{det}^{[2]}\,M \leq 0$, but direct confirmation of this fact is quite nontrivial. $\mathrm{det}^{[2]}\,M$ is a $10$-form in $X_1, \ldots, X_{12}$, having $91$ negative terms and $13$ positive terms. It can be confirmed with semidefinite programming, as described in Section~\ref{secpsatz} below, that $-\mathrm{det}^{[2]}\,M$ does not belong to the smallest cone of polynomials which includes $X_1, \ldots, X_{12}$. However, very easy computations confirm that $J_1, J_2J_3, J_5, J_1J_4-J_2J_3>0$ on $\mathbb{R}^{12}_{\gg 0}$. By a general result on $5$-patterns below (Lemma~\ref{lem5pat3}), $\mathrm{det}^{[2]}\,M\leq 0$.
\end{example}

\begin{example}
\label{ex5CRN}
Consider the following $5\times 5$ matrix of one forms which, as in Example~\ref{ex4a}, arises as the Jacobian matrix of a CRN (shown to the right):
\begin{center}
\begin{tikzpicture}[scale=1.7]

\node at (-0.85,0) {$\mathrm{W}$};
\node at (-0.26,0.81) {$\mathrm{X}$};
\node at (0.72,0.45) {$\mathrm{Y+V}$};
\node at (0.72,-0.45) {$\mathrm{Z+W}$};
\node at (-0.26,-0.81) {$\mathrm{V}$};

\draw[->, thick] (-0.77,0.12) -- (-0.34,0.69);
\draw[<-, thick] (-0.77,-0.12) -- (-0.34,-0.69);
\draw[->, thick] (-0.11,0.81) -- (0.55,0.64);
\draw[<-, thick] (-0.11,-0.81) -- (0.55,-0.64);
\draw[->, thick] (0.72,0.25) -- (0.69,-0.25);

\node at (-5,0) {$M=\left(\begin{array}{ccccc}X_1+X_4&-X_6&-X_3&X_5&-X_7\\-X_1-X_4&X_6+X_2&0&-X_5&X_7\\0&-X_2&X_3&0&0\\X_4&0&-X_3&X_5&0\\-X_4&X_6&0&-X_5&X_7\end{array}\right)$};

\end{tikzpicture}
\end{center}
$\mathrm{det}^{[2]}\,M$ is a $10$-form in $X_1, \ldots, X_7$ with 531 terms. The associated set of matrices 
\[
\{M(x_1, \ldots, x_7) \colon x_1, \ldots, x_7 >0\}
\] 
is a subset of a qualitative class, but this does not help as $\mathrm{det}^{[2]}$ for the associated sign pattern is sign-indefinite. However, as in Example~\ref{ex4a}, it can be ascertained without calculation that $\mathrm{det}\,M$ vanishes identically. Thus, according to Proposition~\ref{J2formula}, $\mathrm{det}^{[2]}M$ factorises as $-J_4(J_1^2J_4+J_3(J_1J_2-J_3))$. We easily compute that $J_4>0$, $J_3>0$ and $J_1J_2-J_3>0$, and thus $\mathrm{det}^{[2]}M>0$. 

\end{example}

\begin{example}
Let $M$ be any $n$-pattern such that $G_{\overline{M}}$ is bipartite, for example:

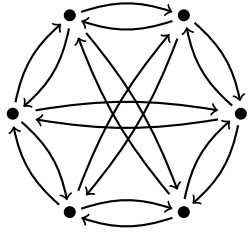
\begin{center}
\begin{tikzpicture}[scale=1.5]

\fill[color=black] (-1,0) circle (1.5pt);
\fill[color=black] (1,0) circle (1.5pt);
\fill[color=black] (-0.5,0.87) circle (1.5pt);
\fill[color=black] (0.5,0.87) circle (1.5pt);
\fill[color=black] (-0.5,-0.87) circle (1.5pt);
\fill[color=black] (0.5,-0.87) circle (1.5pt);

\begin{scope}[xshift=-1cm,rotate=-60]
\draw [->, thick] (0.1,0.03) .. controls (0.4,0.13) and (0.6,0.13) .. (0.9,0.03);
\draw [->, thick] (0.9,-0.05) .. controls (0.6,-0.15) and (0.4,-0.15) .. (0.1,-0.05);
\end{scope}

\begin{scope}[xshift=-1cm,rotate=60]
\draw [->, thick] (0.1,0.03) .. controls (0.4,0.13) and (0.6,0.13) .. (0.9,0.03);
\draw [->, thick] (0.9,-0.05) .. controls (0.6,-0.15) and (0.4,-0.15) .. (0.1,-0.05);
\end{scope}

\begin{scope}[xshift=1cm,rotate=120]
\draw [->, thick] (0.1,0.03) .. controls (0.4,0.13) and (0.6,0.13) .. (0.9,0.03);
\draw [->, thick] (0.9,-0.05) .. controls (0.6,-0.15) and (0.4,-0.15) .. (0.1,-0.05);
\end{scope}

\begin{scope}[xshift=1cm,rotate=-120]
\draw [->, thick] (0.1,0.03) .. controls (0.4,0.13) and (0.6,0.13) .. (0.9,0.03);
\draw [->, thick] (0.9,-0.05) .. controls (0.6,-0.15) and (0.4,-0.15) .. (0.1,-0.05);
\end{scope}

\begin{scope}[xshift=-0.5cm,yshift=0.87cm]
\draw [->, thick] (0.1,0.03) .. controls (0.4,0.13) and (0.6,0.13) .. (0.9,0.03);
\draw [->, thick] (0.9,-0.05) .. controls (0.6,-0.15) and (0.4,-0.15) .. (0.1,-0.05);
\end{scope}

\begin{scope}[xshift=-0.5cm,yshift=-0.87cm]
\draw [->, thick] (0.1,0.03) .. controls (0.4,0.13) and (0.6,0.13) .. (0.9,0.03);
\draw [->, thick] (0.9,-0.05) .. controls (0.6,-0.15) and (0.4,-0.15) .. (0.1,-0.05);
\end{scope}

\begin{scope}[xshift=-1cm]
\draw [->, thick] (0.2,0.03) .. controls (0.8,0.13) and (1.2,0.13) .. (1.8,0.03);
\draw [->, thick] (1.8,-0.05) .. controls (1.2,-0.15) and (0.8,-0.15) .. (0.2,-0.05);
\end{scope}

\begin{scope}[xshift=-0.5cm,yshift=-0.87cm,rotate=60]
\draw [->, thick] (0.2,0.03) .. controls (0.8,0.13) and (1.2,0.13) .. (1.8,0.03);
\draw [->, thick] (1.8,-0.05) .. controls (1.2,-0.15) and (0.8,-0.15) .. (0.2,-0.05);
\end{scope}
 
\begin{scope}[xshift=0.5cm,yshift=-0.87cm,rotate=120]
\draw [->, thick] (0.2,0.03) .. controls (0.8,0.13) and (1.2,0.13) .. (1.8,0.03);
\draw [->, thick] (1.8,-0.05) .. controls (1.2,-0.15) and (0.8,-0.15) .. (0.2,-0.05);
\end{scope}

\end{tikzpicture}
\end{center}
Then $\mathrm{det}^{[2]}\,M = 0$. This is the claim in Proposition~\ref{propgen}(2) below.
\end{example}

\begin{example}
\label{exobstruct1a}
Let $M$ be any $6$-pattern such that $G_{\overline{M}}$ includes a subgraph isomorphic to
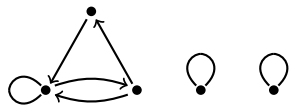
\begin{center}
\begin{tikzpicture}[scale=1.2]

\fill[color=black] (0,0) circle (1.5pt);
\fill[color=black] (-0.5,-0.87) circle (1.5pt);
\fill[color=black] (0.5,-0.87) circle (1.5pt);
\fill[color=black] (1.2,-0.87) circle (1.5pt);
\fill[color=black] (2,-0.87) circle (1.5pt);

\draw[->, thick] (-0.05,-0.087) -- (-0.45,-0.8); 
\draw[<-, thick] (0.05,-0.087) -- (0.45,-0.8); 
\begin{scope}[xshift=-0.5cm,yshift=-0.87cm]
\draw [-, thick] (-0.05,0.05) .. controls (-0.25,0.25) and (-0.4,0.1) .. (-0.4,0);
\draw [-, thick] (-0.05,-0.05) .. controls (-0.25,-0.25) and (-0.4,-0.1) .. (-0.4,0);
\end{scope}

\begin{scope}[xshift=-0.5cm,yshift=-0.87cm]
\draw [->, thick] (0.1,0.05) .. controls (0.4,0.15) and (0.6,0.15) .. (0.9,0.05);
\draw [->, thick] (0.9,-0.05) .. controls (0.6,-0.15) and (0.4,-0.15) .. (0.1,-0.05);
\end{scope}

\begin{scope}[xshift=1.2cm, yshift=-0.87cm]
\draw [-, thick] (0.05,0.05) .. controls (0.25,0.25) and (0.1,0.4) .. (0,0.4);
\draw [-, thick] (-0.05,0.05) .. controls (-0.25,0.25) and (-0.1,0.4) .. (0,0.4);
\end{scope}

\begin{scope}[xshift=2cm, yshift=-0.87cm]
\draw [-, thick] (0.05,0.05) .. controls (0.25,0.25) and (0.1,0.4) .. (0,0.4);
\draw [-, thick] (-0.05,0.05) .. controls (-0.25,0.25) and (-0.1,0.4) .. (0,0.4);
\end{scope}

\end{tikzpicture}
\end{center}
Then $\mathrm{det}^{[2]}\,M \gtrless 0$. This follows from the discussion of ``obstructions'' in Section~\ref{secobstruct} below.

\end{example}

\section{Preliminaries on polynomials}
\label{secpoly}

In order to explore the consequences of Proposition~\ref{J2formula} we require a little additional material on the positivity of polynomials. Given $X = (X_1, \ldots, X_k)$, and a nonnegative integer vector $\gamma = (\gamma_1, \ldots, \gamma_k)$, we denote, as usual, the monomial $\prod X_i^{\gamma_i}$ by $X^{\gamma}$. 

\subsection{Newton polytopes}
Consider a polynomial $p=\sum c_iX^{\alpha_i}$ with exponent vectors $\alpha_i \in \mathbb{Z}^k_{\geq 0}$, and coefficients $c_i$ (assumed all nonzero). If $p'$ is obtained from $p$ by setting some coefficients of $p$ to zero we refer to $p'$ as a {\em subpolynomial} of $p$ and write $p' \subseteq p$. The {\em Newton polytope} of $p$, denoted $\mathrm{N}(p)$, is the convex hull of the exponent vectors $\{\alpha_i\}$ \cite{Sturmfels02solvingsystems}. Given a face $F$ of $N(p)$, we define the associated ``face polynomial'' $p_F := \sum_{\alpha_i \in F}c_iX^{\alpha_i} \subseteq p$. When $F$ is a vertex of $N(p)$ the associated face polynomial consists of a single term. Monomials (resp., terms) of $p$ associated with vertices of $\mathrm{N}(p)$ are referred to as vertex monomials (resp., vertex terms). A polynomial with both positive and negative vertex terms will be referred to as having {\em mixed vertices}. 

\begin{remark}
\label{remsubpoly}
Observe that if $p' \subseteq p$, then non-vertices of $p'$ must necessarily correspond to non-vertices of $p$. Observe also that changing the sign of a term in a polynomial does not cause any change in the nature of this term as a vertex or non-vertex term. 
\end{remark}

The following is well-known and useful. As the proof is brief and seems hard to find, it is given.
\begin{lemma}
\label{lemsgnindef1}
Let $p = \sum c_iX^{\alpha_i} \in \mathbb{R}[X]$. Then $p \geq 0$ on $\mathbb{R}^k_{\gg 0}$ if and only if $p_F\geq 0$ on $\mathbb{R}^k_{\gg 0}$ for each face $F$ of $N(p)$. 
\end{lemma}
\begin{proof} 
As $p_{N(p)}=p$, we only need to prove that if $p \geq 0$ on $\mathbb{R}^k_{\gg 0}$ then $p_F\geq 0$ on $\mathbb{R}^k_{\gg 0}$ for any face $F$ of $N(p)$ of dimension $0 \leq d \leq k-1$.  Let $H$ be a supporting hyperplane of $N(p)$ at $F$ defined by $H := \{x \in \mathbb{R}^k\colon v \cdot x = s\}$ for some $v \in \mathbb{R}^k$ and some $s \in \mathbb{R}$. By definition of a supporting hyperplane, $v \cdot \alpha_i = s$ for $\alpha_i \in F$, and $v \cdot \alpha_i < s$ for $\alpha_i \in N(p)\backslash F$. Suppose $p_F(y) < 0$ for some $y \in \mathbb{R}^k_{\gg 0}$. Consider the curve $\gamma \colon (0, \infty) \to \mathbb{R}^k_{\gg 0}$ defined by $(\gamma(t))_i = y_i e^{v_i t}$. Then
\[
p(\gamma(t)) 
= \sum c_iy^{\alpha_i}e^{(v\cdot \alpha_i)t} = \sum_{\alpha_i \in F}c_iy^{\alpha_i}e^{st} + \sum_{\alpha_i \not \in F}c_iy^{\alpha_i}e^{(v\cdot \alpha_i)t}\,.
\]
So
\[
\lim_{t \to \infty}\frac{p(\gamma(t))}{e^{st}} = \sum_{\alpha_i \in F}c_iy^{\alpha_i} + \lim_{t \to \infty}\sum_{\alpha_i \not \in F}c_iy^{\alpha_i}e^{(v\cdot \alpha_i - s)t} = p_F(y) < 0\,.
\]
The last equality follows as $v\cdot \alpha_i < s$ when $\alpha_i \not \in F$, and so $\lim_{t \to \infty}\sum_{\alpha_i \not \in F}c_iy^{\alpha_i}e^{(v\cdot \alpha_i - s)t} = 0$. Clearly, for sufficiently large $t$, $p(\gamma(t))<0$, a contradiction. 
\end{proof}

The next lemma is an immediate corollary of Lemma~\ref{lemsgnindef1} when we note that a polynomial with a single term which is nonnegative on $\mathbb{R}^k_{\gg 0}$ must in fact be positive on $\mathbb{R}^k_{\gg 0}$. It also follows, for example, from a more general claim in Proposition~1 in \cite{panteakoeppelcraciun}.
\begin{lemma}
\label{lemsgnindef}
If $p \in \mathbb{R}[X]$ has mixed vertices, then $p \gtrless 0$ on $\mathbb{R}^k_{\gg 0}$.
\end{lemma}

The converse of Lemma~\ref{lemsgnindef} is, of course, false: a polynomial may be sign-indefinite without mixed vertices (e.g., $X^2-3XY+Y^2$). Whether this can occur when $M$ is a sign pattern and the polynomial is $\mathrm{det}^{[2]}\,M$, is an interesting question whose answer seems currently unknown (see Remark~\ref{remverts} below).

\begin{remark}
\label{remSNS}
One consequence of Lemma~\ref{lemsgnindef} is that $\mathrm{det}\,M$ for an $n$-pattern $M$ is algebraically uninteresting. Suppose that $\mathrm{det}\,M \neq 0$ and $M$ has $k$ nonzero entries. Consider some monomial of $\mathrm{det}\,M$, say $X^{\alpha}$. As $X^{\alpha}$ corresponds to a particular matching of rows and columns of $M$, $\alpha$ lies on an $n$-dimensional face of the nonnegative orthant in $\mathbb{R}^k$, and is the unique exponent vector of $\mathrm{det}\,M$ on this face. It follows that $X^{\alpha}$ is a vertex monomial of $\mathrm{det}\,M$. Thus, all the terms of $\mathrm{det}\,M$ are vertex terms and, by Lemma~\ref{lemsgnindef}, $\mathrm{det}\,M$ is positive (resp., negative) if and only if all of its terms are positive (resp., negative).
\end{remark}

\subsection{The Positivstellensatz}
\label{secpsatz}
We refer the reader to \cite{BCR}, Theorem~4.2.2, for a complete statement of the Positivstellensatz of Krivine and Stengle, and to \cite{Parrilo03} for further discussion. Here we just state immediate consequences in the special case where the semialgebraic set of interest is the positive orthant. 
\begin{definition}
\label{defpolycone}
Let $X = (X_1, \ldots, X_k)$. Define $M_X :=\{X^{\alpha}\colon\alpha \in \mathbb{Z}^k_{\geq 0}\} \subseteq \mathbb{R}[X]$ to be the multiplicative monoid generated by $\{X_1, \ldots, X_k\}$, namely, the set of all monomials in $X$, including $1$. Define $S \subseteq \mathbb{R}[X]$ to be the squares in $\mathbb{R}[X]$, namely, $S := \{p \in \mathbb{R}[X]\colon p = q^2 \,\,\mbox{ for some }\,\, q \in \mathbb{R}[X]\}$. Define $C_X \subseteq \mathbb{R}[X]$ to be the smallest cone in $\mathbb{R}[X]$ containing $\{X_1, \ldots, X_k\}$, namely,
\[
C_X:= \{p \in \mathbb{R}[X]\colon p = \sum_{i=1}^ns_im_i, \,\,s_i \in S\,\mbox{ and }\, m_i \in M_X\}\,.
\]
Finally, let $C_X'$ consist of those elements of $C_X$ such that $s_i$ can be chosen to be a nonzero constant for at least one $i$. 
\end{definition}

\begin{lemma}
Let $X = (X_1, \ldots, X_k)$, and let $0 \neq h \in \mathbb{R}[X]$. Then
\begin{enumerate}
\item $h>0$ (resp., $h<0$) if and only if there exists $p \in C_X$ (resp., $p \in -C_X$) such that $ph \in C_X'$. 
\item $h\geq 0$ (resp., $h \leq 0$) if and only if there exists $p \in C_X$ (resp., $p \in -C_X$) such that $ph \in C_X\backslash\{0\}$.
\end{enumerate}
\end{lemma}

\begin{proof}
We treat the cases where $h>0$ or $h\geq 0$; the cases where $h<0$ or $h \leq 0$ are similar. 
\begin{enumerate}
\item ($\Leftarrow$) Suppose there exists $p \in C_X$ such that $ph =q \in C_X'$. Then, since $q>0$ and $p \geq 0$, it follows that $p>0$ and, consequently, $h=q/p>0$. ($\Rightarrow$) Suppose that $h > 0$, namely that the system $x_i \geq 0$, $-h(x)\geq 0$, $x_i \neq 0$ is infeasible. In this case, according to the Positivstellensatz, there exist squares $s_i, t_j \in S$, and integer vectors $a_i, b_j, b \in \mathbb{Z}^k_{\geq 0}$ such that:
\[
h(X)\left[\sum_j t_j(X) X^{b_j}\right] = \left[\sum_i s_i(X) X^{a_i}\right] + X^{2b}\,.
\]
The LHS is clearly of the form $ph$ where $p \in C_X$, while the RHS clearly belongs to $C_X'$. 

\item ($\Leftarrow$) Suppose there exist $p \in C_X$, $q \in C_X\backslash\{0\}$ such that $ph =q$. Suppose, contrary to the claim, that $h(y)<0$ for some $y \in \mathbb{R}^k_{\gg 0}$. By continuity of $h$, $h(y') < 0$ for all $y'$ in some open neighbourhood of $U$ of $y$ in $\mathbb{R}^k_{\gg 0}$; consequently $ph\leq 0$ on $U$. Since $ph = q \neq 0$ on $U$,  $(ph)(y') < 0$ for some $y' \in U$, a contradiction. ($\Rightarrow$) Suppose $h$ is nonnegative, namely, the system $x_i \geq 0$, $x_i \neq 0$, $-h(x) \geq 0$, $h(x) \neq 0$ (or, more briefly, $x_i >0$, $-h(x) > 0$) is infeasible. In this case, according to the Positivstellensatz, there exist squares $s_i, t_j \in S$, integer vectors $a_i, b_j, b \in \mathbb{Z}^k_{\geq 0}$, and $b' \in \mathbb{Z}_{\geq 0}$ such that:
\[
h(X)\left[\sum_j t_j(X) X^{b_j}\right] =\left[\sum_i s_i(X) X^{a_i}\right] + h^{2b'}X^{2b}\,.
\]
Clearly the LHS is of the form $ph$ where $p \in C_X$, while the RHS is in $C_X \backslash\{0\}$.

\end{enumerate}
This completes the proof.
\end{proof}

Following \cite{Parrilo03}, Theorem~5.1, the problem of deciding whether a given semialgebraic system is feasible can be reduced to a family of semidefinite programs (SDPs) which can be implemented using various packages -- in our case CSDP \cite{csdp} was used. So, for example, the problem of deciding if a given polynomial $h$ belongs to the cone $C_X$ of Definition~\ref{defpolycone} is an SDP, and the same is true for deciding if any polynomial multiple $ph$ belongs to $C_X$ for $0 \neq p$ up to some fixed degree. If the problem is infeasible, then a certificate to this effect is returned by the SDP. For example, the claim of Example~\ref{exharder} that $-\mathrm{det}^{[2]}\,M$ does not belong to $C_X$ was verified in this way. 

On the other hand, for a given positive integer $r$ if there exist $p\neq 0$ of degree $\leq r$ and $q \in C_X$ such that $ph = q$, then SDP will (subject to practical limitations) return a certificate to this effect. This certicate can be used to explicity construct $p$ and $q$ and sometimes, perhaps with some trial and error, we can find $p$ and $q$ with integer coefficients: the polynomials appearing in the proofs of Lemmas~\ref{lem5pat2}~and~\ref{lem5pat3} below were obtained in this way. 

\section{Sign patterns: basic theory}
\label{secsignpat}

\subsection{Combinatorial structure of sign patterns}

Given an $n$-pattern $M$, changing the signs of some entries in $M$ is referred to as {\em re-signing} $M$. The same terminology is used when changing some edge-signs in $G_M$. An equivalence class of sign patterns under re-signing is termed a {\em zero pattern}. Zero patterns are associated with ordinary digraphs, rather than signed digraphs. Given a sign pattern $M$ with signed digraph $G_M$, $\overline{M}$ (resp., $G_{\overline M}$) refers to the associated zero pattern (resp., unsigned digraph).

{\bf Isomorphism and sub-patterns}. Here, an $m$-pattern $M$ and an $n$-pattern $N$ are termed {\em isomorphic} if $m=n$ and there is a vertex-relabelling of $G_N$ which gives either $G_M$ or $G_{M^T}$. The same notion of isomorphism extends to zero patterns. An equivalence classes of isomorphic sign patterns (resp., zero patterns) is termed an unlabelled sign pattern (resp., zero pattern). Given $n$-patterns $M$ and $N$, $M$ is a {\em sub-pattern} of $N$, written $M \leq N$, if $M$ is obtained by replacing some subset of the indeterminates in $N$ with zeros. $M \leq N$ implies that (i) $\mathrm{det}^{[2]}\,M \subseteq \mathrm{det}^{[2]}\,N$; (ii) $\mathcal{Q}_M$ is a subset of the closure of $\mathcal{Q}_N$, denoted $\mathrm{cl}\,\mathcal{Q}_N$; and (iii) that $G_M$ is a subgraph of $G_N$. It is clear that $\leq$ is a partial order on sign patterns and associated objects, both labelled and unlabelled.

{\bf Cycles, hoopings and weak reversibility}. A {\em cycle} in a digraph $G$ means a directed cycle. Cycles of length $n$ are termed $n$-cycles. $1$-cycles are termed {\em loops}, while $3$-cycles are termed {\em triangles}. A set of cycles in a digraph are {\em coincident} if they share a vertex, and {\em disjoint} otherwise. Following \cite{soule}, a {\em hooping} on $V' \subseteq V(G)$ is a union of pairwise disjoint cycles of $G$ covering $V'$. Hoopings correspond to permutations of $V'$. Borrowing terminology from chemical reaction network theory (\cite{hornjackson} for example), a digraph $G$ is {\em weakly reversible} if each of its connected components is strongly connected. Equivalently, each arc of $G$ belongs to some cycle. A sign pattern or zero pattern whose associated digraph is weakly reversible will be termed weakly reversible.

{\bf Minor-sums}. A set of edges $E$ in a signed digraph is referred to as {\em odd} (resp., {\em even}) if $E$ includes an odd (resp., even) number of positive edges. We write $\mathrm{par}(E)=-1$ when $E$ is odd and $\mathrm{par}(E)=1$ when $E$ is even. A cycle is odd (resp., even) if its edge-set is odd (resp., even). Given an $n \times n$ matrix $M$, let $J_i$ ($i=1, \ldots, n$) be its $i$th minor-sum. Terms in $J_i$ are in one-to-one correspondence with hoopings on subsets of the vertices of $G_M$ of size $i$. If $M$ is an $n$-pattern, a hooping on $i$ vertices involving edges $E$ and associated with a permutation $\sigma$ corresponds to a term in $J_i$ of sign $\mathrm{par}(E)\mathrm{sgn}(\sigma)(-1)^i$. In particular, odd (resp., even) cycles correspond to positive (resp., negative) terms in minors. Some further detail and more general discussion are in \cite{banajirutherford1}: the case $k=1$ in that paper corresponds to the situation here. 

\begin{remark}
\label{remWR}
Entries in $M$ corresponding to arcs of $G_M$ which appear in no cycles do not appear in any $J_i$, and hence do not appear in $\mathrm{det}^{[2]}\,M$. Thus, when discussing $\mathrm{det}^{[2]}\,M$, we may disregard such entries and focus on the maximal weakly reversible sub-pattern of $M$. 
\end{remark}

\subsection{{$[2]$}-positivity and related notions} 
An $n$-pattern $M$ is termed ``$[2]$-positive'' if $\mathrm{det}^{[2]}\,M \geq 0$. The notions of $[2]$-negative, $[2]$-zero, $[2]$-nonzero, $[2]$-nonnegative, and $[2]$-nonpositive are defined analogously. An $n$-pattern which is $[2]$-negative or $[2]$-positive is $[2]$-definite; one which is $[2]$-nonnegative or $[2]$-nonpositive is $[2]$-semidefinite; and one which is neither $[2]$-nonnegative nor $[2]$-nonpositive is $[2]$-indefinite. 

\begin{lemma}[Inheritance]
\label{leminherit}
Let $M, N$ be $n$-patterns with $M \leq N$. (i) If $\mathrm{det}^{[2]}\,M \not \leq 0$, then $\mathrm{det}^{[2]}\,N \not \leq 0$. (ii) If $\mathrm{det}^{[2]}\,M \not \geq 0$, then $\mathrm{det}^{[2]}\,N \not \geq 0$. (iii) If $\mathrm{det}^{[2]}\,N = 0$, then $\mathrm{det}^{[2]}\,M = 0$. (iv) If $\mathrm{det}^{[2]}\,M \gtrless 0$, then $\mathrm{det}^{[2]}\,N \gtrless 0$. 
\end{lemma}
\begin{proof}
(i)  If $\mathrm{det}^{[2]}\,M \not \leq 0$, then there exists $M_1 \in \mathcal{Q}_M$ such that $\mathrm{det}^{[2]}\,M_1 > 0$. As $M_1 \in \mathrm{cl}\,\mathcal{Q}_N$, by continuity of $\mathrm{det}^{[2]}$, there exists $N_1 \in \mathcal{Q}_N$ (close to $M_1$) such that $\mathrm{det}^{[2]}\,N_1 > 0$. (ii) follows similarly to (i). (iii) and (iv) are immediate consequences of (i) and (ii).
\end{proof}

\begin{prop}
\label{propgen}
Let $n \geq 2$, and let $M$ be an $n$-pattern.

\begin{enumerate}
\item If $n \equiv 0$ or $1 \pmod 4$, then $\mathrm{det}^{[2]}\,(-M) = \mathrm{det}^{[2]}\,M$. If $n \equiv 2$ or $3 \pmod 4$, then $\mathrm{det}^{[2]}\,(-M) = -\mathrm{det}^{[2]}\,M$.

\item If $G_M$ is bipartite, then $\mathrm{det}^{[2]}\,M = 0$. 

\end{enumerate}
\end{prop}
\begin{proof}
(1) The transformation $M \mapsto -M$ sends $J_i \mapsto -J_i$ for odd $i$ and $J_i \mapsto J_i$ for even $i$. If $n \equiv 0$ or $1 \pmod 4$ (resp., $n \equiv 2$ or $3 \pmod 4$), this changes the sign of an even (resp., odd) number of columns in the determinantal formula of Proposition~\ref{J2formula}, and so $\mathrm{det}^{[2]}\,(-M) = \mathrm{det}^{[2]}\,M$ (resp., $\mathrm{det}^{[2]}\,(-M) = -\mathrm{det}^{[2]}\,M$).

(2) $G_M$ is bipartite if and only if it has no cycles of odd length. Consequently, as any hooping on an odd number of vertices must contain a cycle of odd length, $J_i=0$ for odd $i$, Hence, there is a column of zeros in the determinantal formula of Proposition~\ref{J2formula}, giving $\mathrm{det}^{[2]}\,M = 0$.
\end{proof}

\subsection{Obstructions}
\label{secobstruct}

Examining sign patterns we quickly find that the presence of certain structures has implications for the sign of $\mathrm{det}^{[2]}$. Most basic of such observations is:
\begin{prop}
\label{propgen1}
Let $n \geq 2$, and let $M$ be an $n$-pattern. 

\begin{enumerate}
\item[(i)] If $n\equiv 0 \pmod 8$ and $G_M$ includes an $(n-1)$-cycle, then $\mathrm{det}^{[2]}\,M \nleq 0$.
\item[(ii)] If $n\equiv 4 \pmod 8$ and $G_M$ includes an $(n-1)$-cycle, then $\mathrm{det}^{[2]}\,M \ngeq 0$.
\item[(iii)] If $n\equiv 1 \pmod 8$ and $G_M$ includes an $n$-cycle, then $\mathrm{det}^{[2]}\,M \nleq 0$.
\item[(iv)] If $n\equiv 5 \pmod 8$ and $G_M$ includes an $n$-cycle, then $\mathrm{det}^{[2]}\,M \ngeq 0$.
\end{enumerate}
\end{prop}
\begin{proof}
If $G_M$ consists of an $k$-cycle, then $J_k \neq 0$, while $J_i=0$ for $i \neq k$. Using the formula of Proposition~\ref{J2formula} we have: (i) If $n\equiv 0 \pmod 8$ and $G_M$ consists of an $(n-1)$-cycle, then $\mathrm{det}^{[2]}\,M = J_{n-1}^{n/2} > 0$. (ii) If $n\equiv 4 \pmod 8$ and $G_M$ consists of an $(n-1)$-cycle, then $\mathrm{det}^{[2]}\,M = -J_{n-1}^{n/2} < 0$. (iii) If $n\equiv 1 \pmod 8$ and $G_M$ consists of an $n$-cycle, then $\mathrm{det}^{[2]}\,M = J_{n}^{(n-1)/2}>0$. (iv) If $n\equiv 5 \pmod 8$ and $G_M$ consists of an $n$-cycle, then $\mathrm{det}^{[2]}\,M = -J_{n}^{(n-1)/2} < 0$. The result in each case follows by inheritance (Lemma~\ref{leminherit}(i)~and~(ii)).
\end{proof}

A variety of results along the lines of Proposition~\ref{propgen1} can be found. These amount to finding unsigned digraphs whose presence as subgraphs of $G_{\overline M}$ ensures that $\mathrm{det}^{[2]}\,M \ngeq 0$ or $\mathrm{det}^{[2]}\,M \nleq 0$. We term them {\em obstructions} to $[2]$-nonnegativity and to $[2]$-nonpositivity respectively. For example, according to Proposition~\ref{propgen1}(ii), a triangle is an obstruction to $[2]$-nonnegativity in $4$-patterns. 

When $n \equiv 2$ or $3 \pmod 4$, $\mathrm{det}^{[2]}\,(-M) = -\mathrm{det}^{[2]}\,M$ (Proposition~\ref{propgen}(1)), and so any obstruction to $[2]$-nonnegativity must also be an obstruction to $[2]$-nonpositivity, forcing any $n$-pattern including such an obstruction to be $[2]$-indefinite. For example, consider again the digraph in Example~\ref{exobstruct1a}. If $M$ is a $6$-pattern $M$ such that $G_{\overline M}$ includes a subgraph isomorphic to this one, then $M$ is $[2]$-indefinite, regardless of edge-signs. This can be confirmed by checking that all possible choices of signs lead to $\mathrm{det}^{[2]}\,M$ having mixed vertices for every sign pattern $M$ in this zero pattern and applying Lemmas~\ref{lemsgnindef}~and~\ref{leminherit}(iv). (The process of confirming this claim is most efficiently done with the help of the discussion of hoopings in Section~\ref{secsignpat} above.)

When $n \equiv 0$ or $1 \pmod 4$, obstructions to $[2]$-nonnegativity are not necessarily obstructions to $[2]$-nonpositivity, and vice versa. For example, by Proposition~\ref{J2formula}, a loop coincident with a $4$-cycle is an obstruction to $[2]$-nonnegativity in $5$-patterns (for this pattern, from the discussion on hoopings and minor-sums, $J_2=J_3=J_5=0$ while $J_1$ and $J_4$ are nonzero monomials, and so $q_5 = J_1^2J_4^2 > 0$). A slightly more exotic example in the same spirit is the following.
\begin{example}
\label{ex8pat}
Consider an $8$-pattern $M$ such that $G_{\overline{M}}$ is isomorphic to:
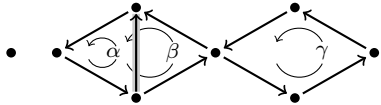
\begin{center}
\begin{tikzpicture}[scale=1.2]

\node at (-0.25,-0.5) {$\scriptstyle{\alpha}$};
\node at (0.4,-0.5) {$\scriptstyle{\beta}$};
\node at (2.04,-0.5) {$\scriptstyle{\gamma}$};

\fill[color=black] (-1.37,-0.5) circle (1.5pt);
\fill[color=black] (0,0) circle (1.5pt);
\fill[color=black] (-0.87,-0.5) circle (1.5pt);
\fill[color=black] (0.87,-0.5) circle (1.5pt);
\fill[color=black] (0,-1) circle (1.5pt);
\fill[color=black] (1.74,0) circle (1.5pt);
\fill[color=black] (2.61,-0.5) circle (1.5pt);
\fill[color=black] (1.74,-1) circle (1.5pt);

\draw [-] (-0.1,-0.62) .. controls (-0.0,-0.82) and (0.3,-0.82) .. (0.4,-0.62);
\draw [<-] (-0.1,-0.38) .. controls (-0.0,-0.18) and (0.3,-0.18) .. (0.4,-0.38);

\begin{scope}[xshift=1.74cm]
\draw [-] (-0.2,-0.62) .. controls (-0.1,-0.82) and (0.2,-0.82) .. (0.3,-0.62);
\draw [<-] (-0.2,-0.38) .. controls (-0.1,-0.18) and (0.2,-0.18) .. (0.3,-0.38);
\end{scope}

\begin{scope}[xshift=-0.4cm, yshift=-0.2cm, scale=0.6]
\draw [-] (-0.2,-0.62) .. controls (-0.1,-0.82) and (0.2,-0.82) .. (0.3,-0.62);
\draw [<-] (-0.2,-0.38) .. controls (-0.1,-0.18) and (0.2,-0.18) .. (0.3,-0.38);
\end{scope}

\draw[->, thick] (-0.087,-0.05) -- (-0.78,-0.45);
\draw[<-, thick] (-0.087,-0.95) -- (-0.78,-0.55);
 
\draw[->, thick] (0.087,-0.95) -- (0.78,-0.55);
\draw[<-, thick] (0.087,-0.05) -- (0.78,-0.45);

\draw[-, black!20, line width=0.1cm] (0,-0.93) -- (0,-0.07);
\draw[->, thick] (0,-0.93) -- (0,-0.07);

\begin{scope}[xshift=1.74cm]
\draw[->, thick] (-0.087,-0.05) -- (-0.78,-0.45);
\draw[<-, thick] (-0.087,-0.95) -- (-0.78,-0.55);
 
\draw[->, thick] (0.087,-0.95) -- (0.78,-0.55);
\draw[<-, thick] (0.087,-0.05) -- (0.78,-0.45);

\end{scope}

\end{tikzpicture}
\end{center}
The cycles in this digraph consist of a triangle (labelled $\alpha$) and a pair of $4$-cycles (labelled $\beta$ and $\gamma$). Abusing terminology, and denoting the signed monomials associated with the three cycles also by $\alpha, \beta$ and $\gamma$, from the discussion on hoopings above we see that $J_1=J_2=J_5=J_6=J_8=0$, while $J_3=\alpha$, $J_4=\beta+\gamma$ and $J_7=\alpha\gamma$. The determinantal formula for $q_8$ reduces to $J_7^2(J_7-J_3J_4)^2 = \alpha^4\gamma^2\beta^2 > 0$. Thus this digraph is an obstruction to $[2]$-nonpositivity in $8$-patterns.
\end{example}

\section{Sign patterns: low dimensional cases}
\label{secsignpat1}

We present some results on $n$-patterns for $2 \leq n \leq 5$. 

$2$-patterns are trivial. For a $2$-pattern $M$, $\mathrm{det}^{[2]}\,M = \mathrm{Tr}\,M$. Consequently $\mathrm{det}^{[2]}\,M=0$ if and only if $G_M$ has no loops; $\mathrm{det}^{[2]}\,M>0$ (resp., $\mathrm{det}^{[2]}\,M<0$) if and only if $G_M$ has a loop and all loops of $G_M$ are odd (resp., even); and $\mathrm{det}^{[2]}\,M \gtrless 0$ if and only if $G_M$ has one odd loop and one even loop. $n=2$ is special, being the only case where the formula for $\mathrm{det}^{[2]}\,M$ does not involve all $J_i$ ($i=1, \ldots, n$).

\subsection{$3$-patterns}

$n=3$ is the first nontrivial case, but it turns out that $3$-patterns have various properties not shared in higher dimensions. Most important of these is the following:
\begin{lemma}
\label{lem3verts}
Let $M$ be a $3$-pattern. Then either $\mathrm{det}^{[2]}\,M >0$, $\mathrm{det}^{[2]}\,M < 0$, $\mathrm{det}^{[2]}\,M = 0$, or $\mathrm{det}^{[2]}\,M \gtrless 0$. Moreover, provided $\mathrm{det}^{[2]}\,M \neq 0$, $\mathrm{det}^{[2]}\,M >0$ (resp., $\mathrm{det}^{[2]}\,M <0$) if and only if all terms of $\mathrm{det}^{[2]}\,M$ are positive (resp., negative). Consequently, $\mathrm{det}^{[2]}\,M \gtrless 0$ if and only if $\mathrm{det}^{[2]}\,M$ has mixed terms.
\end{lemma}
\begin{proof}
Let $M$ be a $3 \times 3$ matrix with $ij$th entry $x_{ij}$. 
Then 
\begin{eqnarray*}
\mathrm{det}^{[2]}\,M &=& 
2x_{11}x_{22}x_{33}+x_{11}^2x_{22}+x_{11}x_{22}^2+x_{11}^2x_{33}+x_{11}x_{33}^2+x_{22}^2x_{33}+x_{22}x_{33}^2\\&&-x_{12}x_{23}x_{31}-x_{13}x_{21}x_{32}-x_{11}x_{12}x_{21}-x_{11}x_{13}x_{31}\\&&-x_{12}x_{21}x_{22}-x_{22}x_{23}x_{32}-x_{13}x_{31}x_{33}-x_{23}x_{32}x_{33}\,.
\end{eqnarray*}
With the help of the {\tt newton\_polytope} function in SageMath \cite{sagemath} we confirm that $x_{11}x_{22}x_{33}$ is the unique non-vertex monomial in $\mathrm{det}^{[2]}\,M$. 
Consider a $3$-pattern  $N \leq M$ obtained by choosing some subset $S$ (possibly empty) of the variables $x_{ij}$, and replacing variables not in $S$ with zeros. $P:=\mathrm{det}^{[2]}\,N \subseteq \mathrm{det}^{[2]}\,M$ is obtained from $\mathrm{det}^{[2]}\,M$ by setting variables not in $S$ to zero. Let $N'$ be obtained from $N$ by re-signing some subset $S' \subseteq S$ (possibly empty) of the variables in $N$ so that $P':=\mathrm{det}^{[2]}\,N'$ is obtained from $P$ by re-signing variables in $S'$. Each possible $3$-pattern upto isomorphism can be obtained in this way. If $P=0$, then $P'=0$. Otherwise:
\begin{enumerate}
\item If $S$ does not include $x_{11}$, $x_{22}$ and $x_{33}$, then all terms in $P'$ are vertex terms (Remark~\ref{remsubpoly}).
\item If $S$ includes $x_{11}$, $x_{22}$ and $x_{33}$ then $x_{11}x_{22}x_{33}$ is the unique non-vertex monomial of $P'$ (Remark~\ref{remsubpoly}). If all vertices of $P'$ are positive (resp., negative) then, in particular, each of $x_{11}x_{22}^2$, $x_{22}x_{33}^2$ and $x_{33}x_{11}^2$ has positive (resp., negative) coefficient in $P'$, and so $x_{11}$, $x_{22}$ and $x_{33}$ do not belong (resp., do belong) to $S'$. Consequently, $x_{11}x_{22}x_{33}$ has positive (resp., negative) coefficient in $P'$. 
\end{enumerate}
In both cases above, $P'$ has a positive (resp., negative) term if and only if it has a positive (resp., negative) vertex term. By Lemma~\ref{lemsgnindef}, $P'$ is either positive, negative or sign-indefinite, being positive (resp., negative, resp., sign-indefinite) if and only if its terms are all positive (resp., all negative, resp., mixed).
\end{proof}

\begin{remark}
\label{remverts}
According to the proof of Lemma~\ref{lem3verts}, for a $3$-pattern $M$, $\mathrm{det}^{[2]}\,M \gtrless 0$ if and only if $\mathrm{det}^{[2]}\,M$ has mixed vertices. Exhaustive search confirms that this also holds true for $4$-patterns. It is currently unknown whether this conclusion holds for an arbitrary $n$-pattern. This question is discussed briefly in the concluding section. 
\end{remark}

A complete combinatorial characterisation of $3$-patterns is fairly straightforward. The following digraphs are needed.
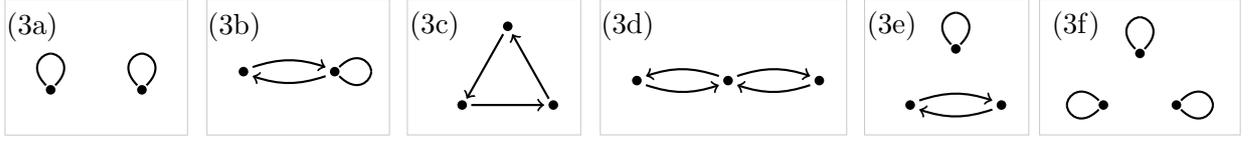
\begin{center}
\begin{tikzpicture}[scale=1.2]
\draw[color=black!20] (-0.5,-0.5) -- (-0.5,1) -- (1.5,1) -- (1.5,-0.5) -- cycle;

\node at (-0.2,0.7) {$\mathrm{(3a)}$};

\fill[color=black] (0,0) circle (1.5pt);
\fill[color=black] (1,0) circle (1.5pt);

\draw [-, thick] (0.05,0.05) .. controls (0.25,0.25) and (0.1,0.4) .. (0,0.4);
\draw [-, thick] (-0.05,0.05) .. controls (-0.25,0.25) and (-0.1,0.4) .. (0,0.4);
\begin{scope}[xshift=1cm]
\draw [-, thick] (0.05,0.05) .. controls (0.25,0.25) and (0.1,0.4) .. (0,0.4);
\draw [-, thick] (-0.05,0.05) .. controls (-0.25,0.25) and (-0.1,0.4) .. (0,0.4);
\end{scope}
\end{tikzpicture}
\begin{tikzpicture}[scale=1.2]
\draw[color=black!20] (-0.4,-0.7) -- (-0.4,0.8) -- (1.6,0.8) -- (1.6,-0.7) -- cycle;

\node at (-0.1,0.5) {$\mathrm{(3b)}$};

\fill[color=black] (0,0) circle (1.5pt);
\fill[color=black] (1,0) circle (1.5pt);

\begin{scope}[xshift=1cm]
\draw [-, thick] (0.05,0.05) .. controls (0.25,0.25) and (0.4,0.1) .. (0.4,0);
\draw [-, thick] (0.05,-0.05) .. controls (0.25,-0.25) and (0.4,-0.1) .. (0.4,0);
\end{scope}

\draw [->, thick] (0.1,0.05) .. controls (0.4,0.15) and (0.6,0.15) .. (0.9,0.05);
\draw [->, thick] (0.9,-0.05) .. controls (0.6,-0.15) and (0.4,-0.15) .. (0.1,-0.05);
\end{tikzpicture}
\begin{tikzpicture}[scale=1.2]
\draw[color=black!20] (-1.1,-1.2) -- (-1.1,0.3) -- (0.8,0.3) -- (0.8,-1.2) -- cycle;

\node at (-0.8,0) {$\mathrm{(3c)}$};

\fill[color=black] (0,0) circle (1.5pt);
\fill[color=black] (-0.5,-0.87) circle (1.5pt);
\fill[color=black] (0.5,-0.87) circle (1.5pt);

\draw[->, thick] (-0.05,-0.087) -- (-0.45,-0.8); 
\draw[<-, thick] (0.05,-0.087) -- (0.45,-0.8); 
\draw[->, thick] (-0.4,-0.87) -- (0.4,-0.87); 

\end{tikzpicture}
\begin{tikzpicture}[scale=1.2]
\draw[color=black!20] (-0.4,-0.6) -- (-0.4,0.9) -- (2.3,0.9) -- (2.3,-0.6) -- cycle;

\node at (-0.1,0.6) {$\mathrm{(3d)}$};

\fill[color=black] (0,0) circle (1.5pt);
\fill[color=black] (1,0) circle (1.5pt);
\fill[color=black] (2,0) circle (1.5pt);

\draw [<-, thick] (0.1,0.05) .. controls (0.4,0.15) and (0.6,0.15) .. (0.9,0.05);
\draw [<-, thick] (0.9,-0.05) .. controls (0.6,-0.15) and (0.4,-0.15) .. (0.1,-0.05);

\begin{scope}[xshift=1cm] 
\draw [->, thick] (0.1,0.05) .. controls (0.4,0.15) and (0.6,0.15) .. (0.9,0.05);
\draw [->, thick] (0.9,-0.05) .. controls (0.6,-0.15) and (0.4,-0.15) .. (0.1,-0.05);
\end{scope}
\end{tikzpicture}
\begin{tikzpicture}[scale=1.2]
\draw[color=black!20] (-1,-1.2) -- (-1,0.3) -- (0.8,0.3) -- (0.8,-1.2) -- cycle;
\node at (-0.7,0) {$\mathrm{(3e)}$};

\fill[color=black] (0,-0.25) circle (1.5pt);
\fill[color=black] (-0.5,-0.87) circle (1.5pt);
\fill[color=black] (0.5,-0.87) circle (1.5pt);

\begin{scope}[yshift=-0.25cm] 
\draw [-, thick] (0.05,0.05) .. controls (0.25,0.25) and (0.1,0.4) .. (0,0.4);
\draw [-, thick] (-0.05,0.05) .. controls (-0.25,0.25) and (-0.1,0.4) .. (0,0.4);
\end{scope}

\begin{scope}[xshift=-0.5cm,yshift=-0.87cm]
\draw [->, thick] (0.1,0.05) .. controls (0.4,0.15) and (0.6,0.15) .. (0.9,0.05);
\draw [->, thick] (0.9,-0.05) .. controls (0.6,-0.15) and (0.4,-0.15) .. (0.1,-0.05);
\end{scope}
\end{tikzpicture}
\begin{tikzpicture}[scale=1.2]
\draw[color=black!20] (-1.1,-1.2) -- (-1.1,0.3) -- (1.1,0.3) -- (1.1,-1.2) -- cycle;
\fill[color=black] (0,-0.3) circle (1.5pt);
\fill[color=black] (-0.4,-0.87) circle (1.5pt);
\fill[color=black] (0.4,-0.87) circle (1.5pt);
\node at (-0.7,0) {$\mathrm{(3f)}$};

\begin{scope}[xshift=0.4cm,yshift=-0.87cm]
\draw [-, thick] (0.05,0.05) .. controls (0.25,0.25) and (0.4,0.1) .. (0.4,0);
\draw [-, thick] (0.05,-0.05) .. controls (0.25,-0.25) and (0.4,-0.1) .. (0.4,0);
\end{scope}
\begin{scope}[yshift=-0.3cm]
\draw [-, thick] (0.05,0.05) .. controls (0.25,0.25) and (0.1,0.4) .. (0,0.4);
\draw [-, thick] (-0.05,0.05) .. controls (-0.25,0.25) and (-0.1,0.4) .. (0,0.4);
\end{scope}
\begin{scope}[xshift=-0.4cm,yshift=-0.87cm]
\draw [-, thick] (-0.05,0.05) .. controls (-0.25,0.25) and (-0.4,0.1) .. (-0.4,0);
\draw [-, thick] (-0.05,-0.05) .. controls (-0.25,-0.25) and (-0.4,-0.1) .. (-0.4,0);
\end{scope}

\end{tikzpicture}
\end{center}

\begin{prop}
\label{prop33}
Let $M$ be a $3$-pattern, and $G_M$ the associated digraph. Then 

(i) $\mathrm{det}^{[2]}\,M=0$ if and only if $G_{\overline M}$ has no subgraph isomorphic to (3a), (3b) or (3c). \\
(ii) $\mathrm{det}^{[2]}\,M\gtrless 0$ if and only if $G_M$ contains a subgraph isomorphic to one of the following:
\begin{enumerate}
\item[(S1)] A pair of loops of opposite parity (Figure~\ref{fig33}(i));
\item[(S2)] A pair of loops and an even $2$-cycle (in any configuration, e.g., Figure~\ref{fig33}(ii),~(iii));
\item[(S3)] A pair of $2$-cycles of opposite parity and a loop coincident with both $2$-cycles (e.g., Figure~\ref{fig33}(iv));
\item[(S4)] A pair of loops and a triangle, all of the same parity (e.g., Figure~\ref{fig33}(v)).
\item[(S5)] A coincident loop $C_1$ and $2$-cycle $C_2$, and a triangle $C_3$ with the opposite parity to $C_1\cup C_2$ (e.g., Figure~\ref{fig33}(vi)).
\item[(S6)] Two triangles of opposite parity (e.g., Figure~\ref{fig33}(vii)).
\end{enumerate}

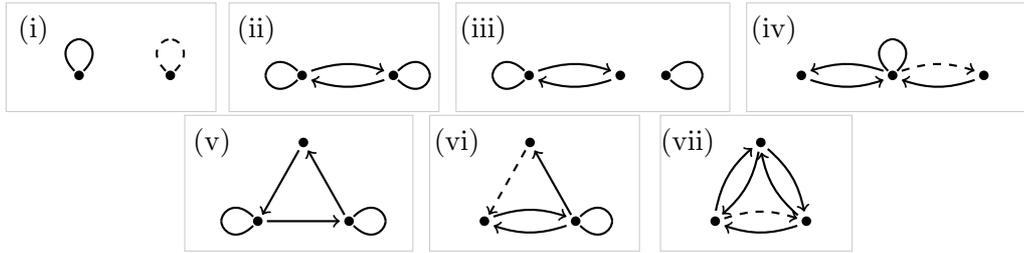
\begin{figure}[h]
\begin{center}
\begin{tikzpicture}[scale=1.2]
\draw[color=black!20] (-0.8,-0.4) -- (-0.8,0.8) -- (1.5,0.8) -- (1.5,-0.4) -- cycle;

\node at (-0.5,0.5) {$\mathrm{(i)}$};

\fill[color=black] (0,0) circle (1.5pt);
\fill[color=black] (1,0) circle (1.5pt);

\draw [-, thick] (0.05,0.05) .. controls (0.25,0.25) and (0.1,0.4) .. (0,0.4);
\draw [-, thick] (-0.05,0.05) .. controls (-0.25,0.25) and (-0.1,0.4) .. (0,0.4);
\begin{scope}[xshift=1cm]
\draw [-, thick, dashed] (0.05,0.05) .. controls (0.25,0.25) and (0.1,0.4) .. (0,0.4);
\draw [-, thick, dashed] (-0.05,0.05) .. controls (-0.25,0.25) and (-0.1,0.4) .. (0,0.4);
\end{scope}
\end{tikzpicture}
\begin{tikzpicture}[scale=1.2]
\draw[color=black!20] (-0.8,-0.4) -- (-0.8,0.8) -- (1.5,0.8) -- (1.5,-0.4) -- cycle;

\node at (-0.5,0.5) {$\mathrm{(ii)}$};

\fill[color=black] (0,0) circle (1.5pt);
\fill[color=black] (1,0) circle (1.5pt);

\draw [-, thick] (-0.05,0.05) .. controls (-0.25,0.25) and (-0.4,0.1) .. (-0.4,0);
\draw [-, thick] (-0.05,-0.05) .. controls (-0.25,-0.25) and (-0.4,-0.1) .. (-0.4,0);
\begin{scope}[xshift=1cm]
\draw [-, thick] (0.05,0.05) .. controls (0.25,0.25) and (0.4,0.1) .. (0.4,0);
\draw [-, thick] (0.05,-0.05) .. controls (0.25,-0.25) and (0.4,-0.1) .. (0.4,0);
\end{scope}

\draw [->, thick] (0.1,0.05) .. controls (0.4,0.15) and (0.6,0.15) .. (0.9,0.05);
\draw [->, thick] (0.9,-0.05) .. controls (0.6,-0.15) and (0.4,-0.15) .. (0.1,-0.05);
\end{tikzpicture}
\begin{tikzpicture}[scale=1.2]
\draw[color=black!20] (-0.8,-0.4) -- (-0.8,0.8) -- (2.2,0.8) -- (2.2,-0.4) -- cycle;
\node at (-0.5,0.5) {$\mathrm{(iii)}$};

\fill[color=black] (0,0) circle (1.5pt);
\fill[color=black] (1,0) circle (1.5pt);
\fill[color=black] (1.5,0) circle (1.5pt);

\draw [-, thick] (-0.05,0.05) .. controls (-0.25,0.25) and (-0.4,0.1) .. (-0.4,0);
\draw [-, thick] (-0.05,-0.05) .. controls (-0.25,-0.25) and (-0.4,-0.1) .. (-0.4,0);
\begin{scope}[xshift=1.5cm]
\draw [-, thick] (0.05,0.05) .. controls (0.25,0.25) and (0.4,0.1) .. (0.4,0);
\draw [-, thick] (0.05,-0.05) .. controls (0.25,-0.25) and (0.4,-0.1) .. (0.4,0);
\end{scope}

\draw [->, thick] (0.1,0.05) .. controls (0.4,0.15) and (0.6,0.15) .. (0.9,0.05);
\draw [->, thick] (0.9,-0.05) .. controls (0.6,-0.15) and (0.4,-0.15) .. (0.1,-0.05);
\end{tikzpicture}
\begin{tikzpicture}[scale=1.2]
\draw[color=black!20] (-0.6,-0.4) -- (-0.6,0.8) -- (2.5,0.8) -- (2.5,-0.4) -- cycle;
\node at (-0.3,0.5) {$\mathrm{(iv)}$};

\fill[color=black] (0,0) circle (1.5pt);
\fill[color=black] (1,0) circle (1.5pt);
\fill[color=black] (2,0) circle (1.5pt);

\begin{scope}[xshift=1cm] 
\draw [-, thick] (0.05,0.05) .. controls (0.25,0.25) and (0.1,0.4) .. (0,0.4);
\draw [-, thick] (-0.05,0.05) .. controls (-0.25,0.25) and (-0.1,0.4) .. (0,0.4);
\end{scope}

\draw [<-, thick] (0.1,0.05) .. controls (0.4,0.15) and (0.6,0.15) .. (0.9,0.05);
\draw [<-, thick] (0.9,-0.05) .. controls (0.6,-0.15) and (0.4,-0.15) .. (0.1,-0.05);

\begin{scope}[xshift=1cm] 
\draw [->, thick,dashed] (0.1,0.05) .. controls (0.4,0.15) and (0.6,0.15) .. (0.9,0.05);
\draw [->, thick] (0.9,-0.05) .. controls (0.6,-0.15) and (0.4,-0.15) .. (0.1,-0.05);
\end{scope}

\end{tikzpicture}


\begin{tikzpicture}[scale=1.2]
\draw[color=black!20] (-1.3,-1.2) -- (-1.3,0.3) -- (1.2,0.3) -- (1.2,-1.2) -- cycle;
\node at (-1,0) {$\mathrm{(v)}$};

\fill[color=black] (0,0) circle (1.5pt);
\fill[color=black] (-0.5,-0.87) circle (1.5pt);
\fill[color=black] (0.5,-0.87) circle (1.5pt);

\draw[->, thick] (-0.05,-0.087) -- (-0.45,-0.8); 
\draw[<-, thick] (0.05,-0.087) -- (0.45,-0.8); 
\draw[->, thick] (-0.4,-0.87) -- (0.4,-0.87); 
\begin{scope}[xshift=0.5cm,yshift=-0.87cm]
\draw [-, thick] (0.05,0.05) .. controls (0.25,0.25) and (0.4,0.1) .. (0.4,0);
\draw [-, thick] (0.05,-0.05) .. controls (0.25,-0.25) and (0.4,-0.1) .. (0.4,0);
\end{scope}

\begin{scope}[xshift=-0.5cm,yshift=-0.87cm]
\draw [-, thick] (-0.05,0.05) .. controls (-0.25,0.25) and (-0.4,0.1) .. (-0.4,0);
\draw [-, thick] (-0.05,-0.05) .. controls (-0.25,-0.25) and (-0.4,-0.1) .. (-0.4,0);
\end{scope}

\end{tikzpicture}
\begin{tikzpicture}[scale=1.2]
\draw[color=black!20] (-1.1,-1.2) -- (-1.1,0.3) -- (1.2,0.3) -- (1.2,-1.2) -- cycle;
\node at (-0.8,0) {$\mathrm{(vi)}$};

\fill[color=black] (0,0) circle (1.5pt);
\fill[color=black] (-0.5,-0.87) circle (1.5pt);
\fill[color=black] (0.5,-0.87) circle (1.5pt);

\draw[->, thick, dashed] (-0.05,-0.087) -- (-0.45,-0.8); 
\draw[<-, thick] (0.05,-0.087) -- (0.45,-0.8); 
\begin{scope}[xshift=0.5cm,yshift=-0.87cm]
\draw [-, thick] (0.05,0.05) .. controls (0.25,0.25) and (0.4,0.1) .. (0.4,0);
\draw [-, thick] (0.05,-0.05) .. controls (0.25,-0.25) and (0.4,-0.1) .. (0.4,0);
\end{scope}

\begin{scope}[xshift=-0.5cm,yshift=-0.87cm]
\draw [->, thick] (0.1,0.05) .. controls (0.4,0.15) and (0.6,0.15) .. (0.9,0.05);
\draw [->, thick] (0.9,-0.05) .. controls (0.6,-0.15) and (0.4,-0.15) .. (0.1,-0.05);
\end{scope}
\end{tikzpicture}
\begin{tikzpicture}[scale=1.2]
\draw[color=black!20] (-1.1,-1.2) -- (-1.1,0.3) -- (1,0.3) -- (1,-1.2) -- cycle;
\node at (-0.8,0) {$\mathrm{(vii)}$};

\fill[color=black] (0,0) circle (1.5pt);
\fill[color=black] (-0.5,-0.87) circle (1.5pt);
\fill[color=black] (0.5,-0.87) circle (1.5pt);

\begin{scope}[rotate=-60]
\draw [->, thick] (0.1,0.03) .. controls (0.4,0.13) and (0.6,0.13) .. (0.9,0.03);
\draw [->, thick] (0.9,-0.05) .. controls (0.6,-0.15) and (0.4,-0.15) .. (0.1,-0.05);
\end{scope}

\begin{scope}[rotate=-120]
\draw [->, thick] (0.1,0.03) .. controls (0.4,0.13) and (0.6,0.13) .. (0.9,0.03);
\draw [->, thick] (0.9,-0.05) .. controls (0.6,-0.15) and (0.4,-0.15) .. (0.1,-0.05);
\end{scope}

\begin{scope}[xshift=-0.5cm,yshift=-0.87cm]
\draw [->, thick,dashed] (0.1,0.03) .. controls (0.4,0.13) and (0.6,0.13) .. (0.9,0.03);
\draw [->, thick] (0.9,-0.05) .. controls (0.6,-0.15) and (0.4,-0.15) .. (0.1,-0.05);
\end{scope}
\end{tikzpicture}
\end{center}
\caption{\label{fig33} Examples of subpatterns which lead to $[2]$-indefiniteness in a $3$-pattern.}
\end{figure}

(iii) $\mathrm{det}^{[2]}\,M>0$ (resp., $\mathrm{det}^{[2]}\,M<0$) if and only if 
\begin{itemize}
\item $G_M$ contains none of the subgraphs identified in part (ii) above; and
\item $G_M$ contains a pair of positive (resp., negative) loops; or an odd (resp., even) triangle; or an even (resp., odd) subgraph of the form (3b).
\end{itemize}

\end{prop}
\begin{proof}
We may assume, w.l.o.g., that $G_M$ is weakly reversible (see Remark~\ref{remWR}). 

(i) Observe that for a weakly reversible digraph $G$ on $3$ vertices, the absence in $G$ of a subgraph isomorphic to (3a), (3b) or (3c) is equivalent to $G$ being isomorphic to a subgraph of (3d) or (3e). By easy application of the observations in Section~\ref{secsignpat}, the patterns corresponding to (3a), (3b) and (3c) are not $[2]$-zero, while those corresponding to (3d) and (3e) are $[2]$-zero. The result follows by inheritance (Lemma~\ref{leminherit}(iii)). 

(ii) This can be confirmed by direct computation. The formula $q_3 = J_3-J_1J_2$ (Proposition~\ref{J2formula}) allows us to understand why the unique minimal $[2]$-indefinite $3$-patterns are those with signed digraphs (S1)--(S6) above, as detailed in Remark~\ref{min3mixed} below.

(iii) A $3$-pattern containing none of the subgraphs identified in (ii) cannot be $[2]$-indefinite. If it contains (3a), (3b) or (3c) then, by part (i), it is not $[2]$-zero, and so must be $[2]$-positive or $[2]$-negative, by Lemma~\ref{lem3verts}. Also by Lemma~\ref{lem3verts} such a pattern is $[2]$-positive (resp., $[2]$-negative) if $\mathrm{det}^{[2]}\,M$ contains a positive (resp., negative) term. We quickly arrive at the conclusion. 
\end{proof}

\begin{remark}
\label{min3mixed}For a $3$-pattern $M$ terms in $J_3$ arise from the hoopings (3c), (3e) and (3f). On the other hand terms in $J_1J_2$, arise from (3a), (3b), (3e) and (3f). A monomial associated with (3e) occurs with the same coefficient in $J_3$ and $J_1J_2$, and so does not figure in $q_3$. A monomial associated with (3f) occurs with the same sign in $J_3$ and $J_1J_2$, but with greater magnitude in $J_1J_2$. Define $J_3'' \subseteq J_3$ to consist of terms of $J_3$ corresponding to (3e) and (3f). Define $J_3':=J_3-J_3''$ and $J_{12}:= J_1J_2-J_3''$, so that $q_3 = J_3'-J_{12}$. By the remarks above, $J_3'$ is associated only with triangles, while $J_{12}$ is associated with digraphs (3a), (3b) and (3f). By Lemma~\ref{lem3verts}, a necessary and sufficient condition for $q_3$ to be sign-indefinite is for one of the following to hold:
\begin{enumerate}
\item $J_{12}$ has mixed terms. To find minimal $3$-patterns with this property it suffices to examine triangle-free, weakly reversible digraphs which are not subgraphs of (3d) or (3e). We are left with (S1) -- (S3) above.
\item $J_3'$ has mixed terms. In other words, $G_M$ must include a pair of triangles of opposite parity. The minimal $3$-pattern with this property is (S6) above.
\item Neither $J_3'$ nor $J_{12}$ have mixed terms, but both are nonzero and have the same sign. Since $J_3'$ is nonzero, $G_{\overline{M}}$ must include a triangle. Since $J_{12}$ is nonzero, $G_{\overline{M}}$ must include (3a) or (3b). The minimal $3$-patterns satisfying these criteria, and not already covered above, are (S4) and (S5) above.
\end{enumerate}
\end{remark}

\subsection{$4$-patterns}

We saw that a $3$-pattern must be $[2]$-definite, $[2]$-indefinite or $[2]$-zero. This does not hold in higher dimensions. Consider, for example, any $4$-pattern where $J_1=0$ but $J_3 \neq 0$. In this case, according to Proposition~\ref{J2formula}, $q_4=-J_3^2$, and so $q_4 \leq 0$ does not necessarily imply $q_4<0$ or $q_4=0$. A minimal example is the following.
\begin{example}
\label{exnz}
Consider the $4$-pattern $M$ shown below alongside $G_M$. 
${}$

\begin{center}
\begin{tikzpicture}[scale=1.3]
\fill[color=black] (0,0) circle (1.5pt);
\fill[color=black] (0.87,-0.5) circle (1.5pt);
\fill[color=black] (0.87,0.5) circle (1.5pt);
\fill[color=black] (1.74,0) circle (1.5pt);

\draw[<-, thick] (0.087,0.05) -- (0.8,0.45); 
\draw[->, thick] (0.087,-0.05) -- (0.8,-0.45); 
\draw[->, thick] (0.87,-0.4) -- (0.87,0.4); 
\draw[<-, thick] (1.68,0.05) -- (0.94,0.45); 
\draw[->, thick, dashed] (1.68,-0.05) -- (0.94,-0.45); 

\begin{scope}[xshift=-3cm]
\node at (0,0) {$M = \left(\begin{array}{cccc}0&x_1&0&0\\0&0&x_2&0\\x_3&0&0&x_4\\0&-x_5&0&0\end{array}\right)$};
\end{scope}
\end{tikzpicture}
\end{center}
In this case $\mathrm{det}^{[2]}\,M = -x_2^2(x_4x_5-x_1x_3)^2 \leq 0$, but $\mathrm{det}^{[2]}\,M \not < 0$ and $\mathrm{det}^{[2]}\,M \neq 0$. 
\end{example}

Even if $\mathrm{det}^{[2]}\,M$ is sign-definite, in dimension $4$ and higher it may have mixed terms. An example is the following slight variant on the previous example.

\begin{example}
\label{exmixed}
Consider the $4$-pattern $M$ shown below alongside $G_M$. 

\begin{center}
\begin{tikzpicture}[scale=1.3]
\fill[color=black] (0,0) circle (1.5pt);
\fill[color=black] (0.87,-0.5) circle (1.5pt);
\fill[color=black] (0.87,0.5) circle (1.5pt);
\fill[color=black] (1.74,0) circle (1.5pt);

\draw[<-, thick] (0.087,0.05) -- (0.8,0.45); 
\draw[->, thick] (0.087,-0.05) -- (0.8,-0.45); 
\draw[->, thick] (0.87,-0.4) -- (0.87,0.4); 
\draw[<-, thick] (1.68,0.05) -- (0.94,0.45); 
\draw[->, thick, dashed] (1.68,-0.05) -- (0.94,-0.45); 

\begin{scope}[xshift=-3cm]
\node at (0,0) {$M = \left(\begin{array}{cccc}0&x_1&0&0\\0&0&x_2&0\\x_3&0&0&x_4\\0&-x_5&0&x_6\end{array}\right)$};
\end{scope}
\begin{scope}[xshift=1.74cm]
\draw [-, thick] (0.05,0.05) .. controls (0.25,0.25) and (0.4,0.1) .. (0.4,0);
\draw [-, thick] (0.05,-0.05) .. controls (0.25,-0.25) and (0.4,-0.1) .. (0.4,0);
\end{scope}
\end{tikzpicture}
\end{center}
In this case $\mathrm{det}^{[2]}\,M = -x_2(x_2(x_4x_5-x_1x_3)^2+x_1x_3x_6^3)<0$ although $\mathrm{det}^{[2]}\,M$ has mixed terms. 
\end{example}

A complete analysis of $4$-patterns in the spirit of Proposition~\ref{prop33} is possible. We can find all minimal $[2]$-nonzero $4$-patterns and minimal $[2]$-indefinite $4$-patterns. Our goal here is to illustrate uses of Proposition~\ref{J2formula}, rather than study sign patterns {\em per se}, and so, instead, we use Proposition~\ref{J2formula} in an inductive way to find a combinatorial sufficient condition for a $4$-pattern to be $[2]$-nonnegative. 

\begin{prop}
\label{prop44}
Let $M$ be a $4$-pattern. Suppose that $G_M$ has the following property:
\begin{quote}
(P)\hspace{0.5cm} No triangles, loops and $2$-cycles are odd, $4$-cycles are even. 
\end{quote}
Then $\mathrm{det^{[2]}}\,M \geq 0$.
\end{prop}

\begin{proof}
$G:=G_M$ can be constructed from the arc-less digraph on $4$ vertices $G_0$ as follows:
\begin{enumerate}
\item[(i)] We first add to $G_0$ all the loops of $G$ to get $G_1$;
\item[(ii)] If there exists a pair of disjoint $2$-cycles in $G$ we choose one such pair and add this to $G_1$ to obtain $G_2$; we continue until there are no more pairs of disjoint $2$-cycles to add. In this way we obtain a signed digraph $G_k$. 
\item[(iii)] If there exist any remaining $2$-cycles in $G$, not already in $G_k$, we add these one at a time, obtaining, eventually, $G_n$ which has all the $2$-cycles of $G$.
\item[(iv)] We add in any remaining edges to get $G_{n+1}=G$. 
\end{enumerate}
We thus have a sequence of signed digraphs:
\[
G_0 \leq G_1 \leq \cdots \leq G_k \leq \cdots \leq G_n \leq G_{n+1}=G\,,
\]
and a corresponding sequence of sign patterns
\[
M^0 \leq M^1 \leq \cdots \leq M^k \leq \cdots \leq M^n \leq M^{n+1}=M\,.
\]
Clearly each of $G_0, \ldots, G_{n+1}$ has property $P$ each being a subgraph of $G$. 

Trivially, $\mathrm{det}^{[2]}\,M^0 = 0$, and $\mathrm{det}^{[2]}\,M^1 \geq 0$ as $\mathrm{det}^{[2]}\,M^1$ is just the product of sums of pairs of diagonal entries of $M$ which are all positive or zero (as loops of $G_M$ are odd). We show that $\mathrm{det}^{[2]}$ is an increasing function on $(M^i)$, in the sense that $\mathrm{det}^{[2]}\,M^{i} - \mathrm{det}^{[2]}\,M^{i-1} \geq 0$ for $i=1, \ldots, n+1$. Consequently $\mathrm{det}^{[2]} M\geq 0$. 

Consider a step of the form (ii), namely add a pair of disjoint $2$-cycles to $M^{i}$ on (disjoint) vertex pairs $\alpha$ and $\beta$ to get $M^{i+1}$. Then, letting $J^i_k$ refer to $k$th minor-sum of $M^i$: 
\begin{enumerate}
\item[(a)] $J^{i+1}_1 = J^i_1 = J^i_{1, \alpha}+J^i_{1, \beta}$, where $J^i_{1, \alpha}$ and $J^i_{1, \beta}$ are the subpolynomials of $J^i_1$ associated with loops on $\alpha$ and $\beta$ respectively;
\item[(b)] $J^{i+1}_2= J^i_2+C_\alpha+C_\beta$, where $C_\alpha$ and $C_\beta$ are terms corresponding, respectively, to the $2$-cycles on $\alpha$ and $\beta$;
\item[(c)] $J^{i+1}_3= J^i_3+J^i_{1, \beta}C_\alpha+J^i_{1, \alpha}C_\beta$, as no triangles were created; and 
\item[(d)] $J^{i+1}_4= J^i_4+C_\alpha C_\beta-R$. Here $R$ is either zero in the case that no new $4$-cycles were created, or a nonzero polynomial corresponding to new $4$-cycles. 
\end{enumerate}
A quick calculation using Proposition~\ref{J2formula} and (a)--(d) above gives, for $\mathrm{det}^{[2]}M^{i+1} - \mathrm{det}^{[2]}M^{i}$,
\begin{equation}
\label{eq4induct}
J^i_{1, \alpha}J^i_{1, \beta}(C_\alpha-C_\beta)^2 + (J^i_{1, \beta}C_\alpha+J^i_{1, \alpha}C_\beta)(J^i_1J^i_2-J^i_3) + (C_\alpha J^i_{1, \alpha}+C_\beta J^i_{1, \beta})J^i_3 +(J^i_1)^2R.
\end{equation}
Assumption $P$ implies that $J^i_{1, \alpha}$ and $J^i_{1, \beta}$ are nonnegative (as loops are odd); $C_\alpha$ and $C_\beta$ are positive (as the added $2$-cycles are odd); and $R$ is either zero in the case that no new $4$-cycles were created, or positive in the case that some new, even, $4$-cycles were created. Moreover, the fact that loops and $2$-cycles in $G_i$ are odd and there are no triangles in $G_i$ implies that $0 \leq J^i_3 \leq J^i_1J^i_2$. Thus, clearly (\ref{eq4induct}) is nonnegative. 

A step of the form (iii), where a single odd $2$-cycle is added on, say, vertices $\alpha$, while there is no $2$-cycle on $\beta$, corresponds to setting $C_\beta=0$ in (\ref{eq4induct}), while a step of the form (iv), where additional edges are added without creating any $2$-cycles corresponds to setting $C_\alpha = C_\beta = 0$ in (\ref{eq4induct}). In every case, $\mathrm{det}^{[2]}\,M^{i+1}-\mathrm{det}^{[2]}\,M^{i}\geq 0$. This completes the proof that $\mathrm{det}^{[2]}$ is increasing on the sequence $(M^i)$ and thus that $\mathrm{det}^{[2]}\,M \geq 0$.
\end{proof}

\begin{remark}
From the proof of Proposition~\ref{prop44}, $\mathrm{det}^{[2]}\,M > 0$ if $\mathrm{det}^{[2]}\,M^i>0$ for any $i = 1, \ldots, n$ (e.g., if $G_M$ has three or more loops). The condition in Proposition~\ref{prop44} that $G_M$ includes no triangles is necessary for the conclusion: by Proposition~\ref{propgen1}(ii) a $4$-pattern with a triangle cannot be $[2]$-nonnegative.
\end{remark}

\subsection{$5$-patterns}

The following three lemmas are corollaries of Proposition~\ref{J2formula} and illustrate how the proposition leads to conditions for sign-definiteness or semidefiniteness of a $5$-pattern. We stress that these are merely {\em examples}, and an exhaustive analysis of the case $n=5$ is likely to reveal more such general conditions. 

\begin{lemma}
\label{lem5pat1}
Consider a $5$-pattern $M$ such that $J_4>0$ and either (i) $J_1>0$, $J_5>0$ and $J_2J_3<0$, or (ii) $J_1<0$, $J_5<0$ and $J_2J_3>0$. Then $\mathrm{det}^{[2]}\,M<0$.
\end{lemma}
\begin{proof}
The formula in Proposition~\ref{J2formula} in the case $n=5$ can be written
\[
q_5 = -(J_5 - J_1J_4)^2 +J_2J_3J_5-J_1J_2^2J_5 -J_3^2J_4 + J_1J_2J_3J_4 
\]
from which the result follows by observation. (Indeed, a number of other conditions based on this rewriting of $q_5$ lead to $\mathrm{det}^{[2]}\,M<0$.)
\end{proof}

\begin{lemma}
\label{lem5pat2}
Consider a $5$-pattern $M$ such that $J_1J_2J_3>0$, $J_4>0$ and $J_1(J_5-J_2J_3)>0$. Then $\mathrm{det}^{[2]}\,M\leq 0$.
\end{lemma}
\begin{proof}
Abbreviating $J_5-J_2J_3$ as $J_*$, we can derive from the formula in Proposition~\ref{J2formula}:
\begin{equation}
\label{eqJstar}
q_5 = 3J_1J_2J_3J_4 + 2J_1J_4J_*-J_2J_3J_*-J_*^2-J_1J_2^3J_3-J_1J_2^2J_*-J_1^2J_4^2-J_3^2J_4\,.
\end{equation}
Setting $\alpha = J_1J_2J_3$, $\beta = J_1J_*$, $\gamma = J_1^2J_4$ and $\delta = J_3^2$ in (\ref{eqJstar}) gives:
\[
J_1^2J_3^2q_5 = 3\alpha\gamma\delta + 2\beta\gamma\delta-\alpha\beta\delta-\beta^2\delta-\alpha^3-\alpha^2\beta-\gamma^2\delta-\gamma \delta^2 := P_1(\alpha, \beta, \gamma, \delta)\,.
\]
It is not immediately apparent, $P_1\leq 0$ for positive values of $\alpha, \beta, \gamma$ and $\delta$. However, using semidefinite programming as outlined in Section~\ref{secpsatz}, we derive
\begin{eqnarray*}
-4(\alpha+\beta+\gamma+\delta)P_1 &=&(4\alpha\gamma+4\alpha\beta+4\beta\gamma+3\gamma\delta+3\beta^2)(\alpha-\delta)^2 + 4(\alpha\delta+\beta\delta)(\alpha+\beta-\gamma)^2\\ && \hspace{1cm} +\gamma\delta(3\alpha+2\beta-2\gamma-\delta)^2+(2\alpha^2+\beta\delta+\alpha\beta-2\gamma\delta)^2
\end{eqnarray*}
By the hypotheses, $\alpha, \beta, \gamma, \delta$ and $J_1^2J_3^2$ are all positive. Thus $P_1\leq 0$ and $q_5 = P_1/(J_1^2J_3^2) \leq 0$. 
\end{proof}

\begin{lemma}
\label{lem5pat3}
Consider a $5$-pattern $M$ such that $J_1J_2J_3>0$, $J_1J_5>0$, and $J_1(J_1J_4-J_2J_3)>0$. Then $\mathrm{det}^{[2]}\,M\leq 0$.
\end{lemma}
\begin{proof}
Abbreviating $J_1J_4-J_2J_3$ as $J_*$, we can derive from the formula in Proposition~\ref{J2formula}:
\begin{equation}
\label{eqJstar1}
J_1q_5 = 3J_1J_2J_3J_5 + 2J_1J_5J_*-J_1J_5^2-J_1^2J_2^2J_5-J_1J_2J_3J_*-J_1J_*^2-J_2J_3^3-J_3^2J_*\,.
\end{equation}
Setting $\alpha = J_1J_2J_3$, $\beta = J_1J_*$, $\gamma = J_1J_5$ and $\delta = J_3^2$ in (\ref{eqJstar1}) gives:
\[
J_1^2J_3^2q_5 = 3\alpha \gamma \delta + 2\beta\gamma\delta-\gamma^2\delta-\alpha^2\gamma-\alpha\beta\delta -\beta^2\delta-\alpha\delta^2-\beta\delta^2:= P_2(\alpha,\beta,\gamma,\delta)\,.
\]
$P_2\leq 0$ for positive values of $\alpha, \beta, \gamma$ and $\delta$. To verify this we may confirm that
\begin{eqnarray*}
-4(\alpha+\beta+\gamma+\delta)P_2 &=& (4\beta\gamma+4\alpha\gamma+3\beta\delta+3\gamma^2)(\alpha-\delta)^2+ 4\alpha\delta(\beta+\delta-\gamma)^2\\&&\hspace{-0.5cm}+4\gamma\delta(\alpha+\beta-\gamma)^2+\beta\delta(\alpha+2\beta+\delta-2\gamma)^2+(2\alpha\delta+2\beta\delta-\alpha\gamma-\gamma\delta)^2
\end{eqnarray*}
This formula was obtained using semidefinite programming as outlined in Section~\ref{secpsatz}. The hypotheses imply that $\alpha, \beta, \gamma, \delta$ and $J_1^2J_3^2$ are all positive. Thus $P_2 \leq 0$ and $q_5 = P_2/(J_1^2J_3^2) \leq 0$.
\end{proof}

\section{Further directions}

This paper begins the study of $\mathrm{det}^{[2]}$ for polynomial matrices, with a heavy emphasis on sign patterns. The formula presented in terms of minor-sums has proved useful, but its implications have only been superficially explored. Below are some questions/avenues which seem to be interesting. 

{\bf $\bm{[2]}$-indefiniteness and mixed vertices.} Recalling Remark~\ref{remverts}, consider the following statement about an $n$-pattern $M$:
\begin{quote}
(INDEF2) \hspace{0.5cm} $\mathrm{det}^{[2]}\,M \gtrless 0$ if and only if $\mathrm{det}^{[2]}\,M$ has mixed vertices.
\end{quote}
INDEF2 has been checked for $n \leq 4$ and for a large number of cases where $n=5$ and no counterexamples found. If INDEF2 is true then, supposing that we can calculate $\mathrm{det}^{[2]}\,M$, the problem of determining $[2]$-indefiniteness of $M$ reduces to a problem in convex geometry which can be solved with linear programming. It would be interesting to verify, or find a minimal counterexample to, INDEF2. 

{\bf Minor-sums and cycles.} Given a polynomial matrix $M$, there exist new variables other than minor-sums which can be useful when examining $\mathrm{det}^{[2]}\,M$. In particular, it can be helpful to take terms associated with cycles in some digraph as the variables. For example, given for a $5$-pattern $M$, $G_M$ has up to $89$ cycles, and so
$\mathrm{det}^{[2]}\,M$ can be written as a polynomial in up to $89$ new variables each of which is a signed monomial associated with a cycle in $G_M$. The value of this approach is best seen in relatively sparse examples such as Examples~\ref{ex8pat}. It is also implicit in constructions such as those of Examples~\ref{exnz}~and~\ref{exmixed}, and in the proof of combinatorial results such as Proposition~\ref{prop44}. 

{\bf Obstructions in sign patterns.} The obstructions discussed in Section~\ref{secobstruct} seem worthy of further study. The examples in Section~\ref{secobstruct} illustrate how the determinantal formula in Proposition~\ref{J2formula} can be used to generate such obstructions. When examining a sign pattern $M$ it is natural to begin by asking the purely combinatorial question of whether $G_{\overline{M}}$ contains any known obstructions to $[2]$-nonnegativity or $[2]$-nonpositivity. 

{\bf CRNs.} Examples~\ref{ex4a}~and~\ref{ex5CRN} highlight that problems which involve polynomial matrices other than sign patterns arise in applications. These examples also hint at interesting possible extensions to the theory here. In particular, the study of CRNs quite generally gives rise to polynomial matrices with algebraic dependencies between their entries as a consequence of certain natural factorisations of these matrices. Examination of the factors in such factorisations can tell us about properties of $\mathrm{det}^{[2]}$ as discussed in \cite{abphopf}. These themes remain to be more fully explored.

\end{document}